\documentclass[final]{siamltex}
\usepackage{amsmath,amsfonts,amssymb}
\newtheorem{remark}{Remark}
\newtheorem{assumption}{Assumption}

\usepackage{showkeys}
\usepackage{showlabels}
\usepackage{placeins}

\usepackage{latexsym}
\usepackage{epsfig,overpic}
\usepackage{todonotes}
\usepackage{booktabs}
\usepackage{multirow}
\usepackage{hyperref}
\hypersetup{
  colorlinks= true,
  allcolors = {green!50!black},
  urlcolor  = {red!50!black},
}

\usepackage{siunitx}
\DeclareSIUnit[number-unit-product = {}]\Q{~}
\DeclareSIPrefix\kilo{K}{3}
\DeclareSIPrefix\mega{M}{6}
\DeclareSIPrefix\giga{G}{9}
\DeclareSIPrefix\terra{T}{12}
\newcommand{\numQ}[1]{\SI[zero-decimal-to-integer,scientific-notation = engineering,round-precision=0,exponent-to-prefix = true]{#1}{\Q}\!\!\!}
\newcommand{\numQQ}[1]{\SI[zero-decimal-to-integer,scientific-notation = engineering,round-precision=1,exponent-to-prefix = true]{#1}{\Q}\!\!\!}
\newcommand{\numt}[1]{\num[round-precision=1,round-mode=places, scientific-notation=false]{#1}s}

\usepackage{pgfplotstable, pgfplots}

\usepackage{tikz}
\usetikzlibrary{shapes,arrows,snakes,calendar,matrix,backgrounds,folding,calc,positioning,spy,decorations.pathreplacing}

\newif\ifarxiv
\arxivtrue

\newcommand{\divergence}{\operatorname{div}}
\newcommand{\id}{\operatorname{id}}
\newcommand{\curl}{\operatorname{curl}}
\newcommand{\normal}{n}

\newcommand{\bilinearform}[1]{\mathcal{#1}_h}

\newcommand{\mesh}{\mathcal{T}_h}
\newcommand{\facets}{\mathcal{F}_h}
\newcommand{\facetsint}{\mathcal{F}_h^{\text{int}}}
\newcommand{\facetsext}{\mathcal{F}_h^{\text{ext}}}
\newcommand{\sumoverallelements}{\sum_{T \in \mathcal{T}_h}}

\newcommand{\sumoverallinnerfacets}{\sum_{F \in \facetsint}}
\newcommand{\sumoverallouterfacets}{\sum_{F \in \facetsext}}
\newcommand{\jumpleft}{[\![}
\newcommand{\jumpright}{]\!]}
\newcommand{\facetjump}[1]{\jumpleft #1 \jumpright_F}
\newcommand{\jumpleftDG}{[}
\newcommand{\jumprightDG}{]}
\newcommand{\jumpDG}[1]{\jumpleftDG #1 \jumprightDG_F}

\newcommand{\averageleftDG}{\{}
\newcommand{\averagerightDG}{\}_F}
\newcommand{\averageDG}[1]{\averageleftDG #1 \averagerightDG}
\newcommand{\stab}{\lambda}
\newcommand{\brokenHnormleft}{|\!|\!|}
\newcommand{\brokenHnormright}{|\!|\!|_1}
\newcommand{\brokenHnorm}[1]{\brokenHnormleft #1 \brokenHnormright}
\newcommand{\localnormleft}{|\!|\!|}
\newcommand{\localnormright}{|\!|\!|_T}
\newcommand{\localnorm}[1]{\localnormleft #1 \localnormright}
\newcommand{\localnormIIleft}{|\!|\!|}
\newcommand{\localnormIIright}{|\!|\!|_{T,\ast}}
\newcommand{\localnormII}[1]{\localnormIIleft #1 \localnormIIright}
\newcommand{\brokenHnormIIleft}{|\!|\!|}
\newcommand{\brokenHnormIIright}{|\!|\!|_{1,*}}
\newcommand{\brokenHnormII}[1]{\brokenHnormIIleft #1 \brokenHnormIIright}

\newcommand{\proj}{\Pi}
\newcommand{\facetproj}{\proj_F}

\newcommand{\Recon}{\mathcal{R}_{\VelSymb}}
\newcommand{\ReconHdiv}{\mathcal{R}_{\HdivSymb}}

\newcommand{\rr}{\mathbb{R}}

\newcommand{\HdivSymb}{W}
\newcommand{\HcurlSymb}{S}

\newcommand{\HdivSymbHODC}{W^-}
\newcommand{\FacetSymb}{F}
\newcommand{\PressureSymb}{Q}
\newcommand{\VelSymb}{U}

\newcommand{\HdivVar}{u_{\mathcal{T}}}
\newcommand{\HdivVarT}{u_{T}}
\newcommand{\HdivVarTT}{u_{T'}}
\newcommand{\FacetVar}{u_{\mathcal{F}}}
\newcommand{\FacetVarF}{u_{F}}

\newcommand{\PressureVar}{p_h}
\newcommand{\PressureVarEx}{p}
\newcommand{\ForceVar}{f}
\newcommand{\VelVar}{u_h}
\newcommand{\VelVarEx}{{u}}
\newcommand{\HdivVarEx}{{u}}

\newcommand{\HdivTest}{v_{\mathcal{T}}}
\newcommand{\HdivTestT}{v_{T}}

\newcommand{\FacetTest}{v_{\mathcal{F}}}
\newcommand{\FacetTestF}{v_{F}}

\newcommand{\PressureTest}{q_h}
\newcommand{\PressureTestEx}{q}
\newcommand{\VelTest}{v_h}
\newcommand{\DiscTest}{v_h}
\newcommand{\VelTestEx}{v}

\newcommand{\VelTestB}{w_h}
\newcommand{\PressureTestB}{r_h}

\newcommand{\HdivSpace}{\HdivSymb_h}
\newcommand{\HdivSpaceHODC}{{\HdivSymbHODC_h}}
\newcommand{\HdivSpaceref}{\widehat{\HdivSymb}_h}
\newcommand{\FacetSpace}{\FacetSymb_h}

\newcommand{\PressureSpace}{\PressureSymb_h}
\newcommand{\VelSpace}{\VelSymb_h}

\newcommand{\HcurlSpace}{\HcurlSymb_h}

\newcommand{\ConDiff}{\mathcal{E}^c}

\newcommand{\jacobi}{p}
\newcommand{\intjacobi}{{\hat{p}}}

\newcommand{\hull}{{\textrm{span}}}

\newcommand{\partialsym}[1]{\partial_{#1}}

\newcommand{\kk}{[$\texttt{n}^{\hphantom{\!\ast\!}}|\texttt{t}^{\hphantom{\!\ast\!}}$]}
\newcommand{\kkmo}{[$\texttt{n}^{\hphantom{\!\ast\!}}|\texttt{t}^{\!\ast\!}$]}
\newcommand{\kmokmo}{[$\texttt{n}^{{\!\ast\!}}|\texttt{t}^{{\!\ast\!}}$]}

\newcommand{\Ureg}{U_{\textrm{reg}}}

\allowdisplaybreaks

\title{Hybrid Discontinuous Galerkin methods with relaxed H(div)-conformity for incompressible flows. Part I\thanks
  {
\ifarxiv%
Revision submitted to SINUM. Appendix part of this report is only in arXiv version.
\else%
Revision submitted to the editors on: \today.
\fi%
    }
}
\author{
  Philip L. Lederer\thanks{Institute for Analysis and Scientific Computing, TU Wien, Wien, Austria; email: {\tt \{philip.lederer,joachim.schoeberl\}@tuwien.ac.at}}
  \and
  Christoph Lehrenfeld\thanks{Institute for Numerical and Applied Mathematics, University of G\"ottingen, G\"ottingen, Germany; email: {\tt lehrenfeld@math.uni-goettingen.de}}
  \and
  Joachim Sch{\"o}berl\footnotemark[1] 
}

\begin{document}
\maketitle
\begin{abstract}
We propose a new discretization method for the Stokes equations.
The method is an improved version of the method recently presented in [C. Lehrenfeld, J. Sch\"oberl, \emph{Comp. Meth. Appl. Mech. Eng.}, 361 (2016)]
which is based on an $H(\divergence)$-conforming finite element space and a Hybrid Discontinuous Galerkin (HDG) formulation of the viscous forces.
$H(\divergence)$-conformity results in favourable properties such as pointwise divergence free solutions and pressure-robustness.
However, for the approximation of the velocity with a polynomial degree $k$ it requires unknowns of degree $k$ on every facet of the mesh.
In view of the superconvergence property of other HDG methods, where only unknowns of polynomial degree $k-1$ on the facets are required to obtain an accurate polynomial approximation of order $k$ (possibly after a local post-processing) this is sub-optimal.
The key idea in this paper is to slightly relax the $H(\divergence)$-conformity so that only unknowns of polynomial degree $k-1$ are involved for normal-continuity. This allows for optimality of the method also in the sense of superconvergent HDG methods.
In order not to loose the benefits of $H(\divergence)$-conformity we introduce a cheap reconstruction operator which restores pressure-robustness and pointwise divergence free solutions and suits well to the finite element space with relaxed $H(\divergence)$-conformity. We present this new method, carry out a thorough $h$-version error analysis and demonstrate the performance of the method on numerical examples.
\end{abstract}

\begin{keywords}
  Stokes equations,
  Hybrid Discontinuous Galerkin methods,
  $H(\divergence)$-conforming finite elements,
  pressure robustness
\end{keywords}
\begin{AMS} 
  35Q30, 65N12, 65N22, 65N30
\end{AMS}

\section{Introduction and structure of the paper}
In the recent paper \cite{LS_CMAME_2016} a new finite element discretization method for the incompressible Navier-Stokes equations has been presented. It is based on an efficient time integration scheme which allows to split the Navier-Stokes problem into two simpler subproblems, a Stokes-type problem and a hyperbolic transport problem. For both subproblems specifically taylored methods have been applied.
In this work we turn our attention to the numerical treatment of the Stokes problem in a velocity-pressure formulation:
\begin{equation} \label{eq:navierstokes}
  \left\{
    \begin{array}{r r l @{\hspace*{0.1cm}} c l l}
      - \nu \Delta \VelVarEx & + \nabla \PressureVarEx
      & =
      & \ForceVar \!
      & \mbox{ in } \Omega, \\
      \divergence ( \VelVarEx) &
      & =
      & 0
      & \mbox{ in } \Omega, 
    \end{array} \right.
\end{equation}
with boundary conditions $\VelVarEx=0$ on $\Gamma_D \subset \partial \Omega$.
Here, $\nu = const$ is the kinematic viscosity, ${\VelVarEx}$ the velocity, $\PressureVarEx$ the pressure, and $\ForceVar$ is an external body force.
At the heart of the present contribution lies the modification of the previously considered $H(\divergence)$-conforming hybrid discontinuous Galerkin finite element method resulting in a reduction of the gloablly coupled unknowns occuring in the linear systems. \newline
Beside other purposes, the aim of abandoning $H^1$-conformity and consider an $H(\divergence)$-conforming discontinuous Galerkin (DG) method instead is to find a suitable approximation of the incompressibility constraint. This leads to a discontinuity only for the tangential part of the velocity and was introduced in \cite{cockburn2005locally, cockburn2007note} and for the Brinkman Problem in \cite{stenberg2011}. Nevertheless DG methods are considered to be much more expensive compared to a continuos Galerkin (CG) method when it comes to solving linear systems. This is mainly due to the dramatically increased numbers of globally coupling degrees of freedom.
An approach to overcome this drawback is the concept of \emph{hybridization}.
In a \emph{hybrid} or \emph{hybridized} methods, discretizations are separated into two parts: local problems (with only local approximation spaces) and transmission conditions that are formulated in terms of facet unknowns on the skeleton of the mesh. 
Hence, even more unknowns are introduced on the skeleton but nevertheless this leads to two major benefits. First, the coupling between elements is reduced and secondly, the structure of the coupling allows for static condensation of the element unknowns, see \cite{cockburn2009unified,LS_CMAME_2016,lehrenfeld2010hybrid}. Hybrid discontinuous Galerkin (HDG) methods are also of interest due to their capability of providing a {\it superconvergent} post processing. This step allows to reconstruct interior element unknowns resulting in an accuracy of order $k+1$ in the volume when we use polynomials of order $k$ on the skeleton.
In \cite{LS_CMAME_2016} a similar result was achieved by the usage of {\it projected jumps}. The idea is to use only polynomials of order $k-1$ on the skeleton and introduce a projection on the occurring jump terms. This reduction of the polynomial order has no impact on the accuracy of the method, thus can be seen as an approach to obtain superconvergence. Still, this technique only results in a reduced coupling of the tangential part. This work considers an approach how to reduce also the coupling with respect to the normal component. Where an $H(\divergence)$-conforming method demands normal continuity of the velocity, i.e.  $\jumpDG{ \HdivVar\! \cdot\! n } = 0$, we only impose a relaxed continuity $\facetproj^{k-1} \jumpDG{ \HdivVar\! \cdot\! n } = 0$. This allows facet jumps of the normal components in the highest order modes resulting in discrete divergence free approximations with optimal order of convergence.

\noindent A weak treatment of the incompressibility constraint using mixed finite element methods results in velocity error estimates which dependent on the pressure, i.e. they are not pressure robust, \cite{JLMNR:sirev}. This may result in a bad velocity approximation when the irrotational part of the right hand side is dominant. Under this perspective one also speaks of poor mass conservation, since large velocity errors are accompanied by large divergence errors. In \cite{linke2014role} the problem is analyzed and a reconstruction is presented to retrieve pressure robustness. This reconstruction maps discrete divergence free testfunctions onto exactly divergence free test functions and is applied only on the right hand side. This approach was performed for several elements using discontinuous pressures \cite{linke2014role, blms:2015, Linke:2012, JLMNR:sirev, lmt:2016} and recently continuous pressure elements \cite{LedererSchoeberl2016, lederer:2016}.
For $H(\divergence)$-conforming finite element methods discrete incompressibility implies pointwise incompressibility and render the methods automatically pressure robust.
These methods have been investigated for incompressible flow problems in a.o. \cite{cockburn2007note,kanschat2008energy,kanschat2015multigrid,LS_CMAME_2016,fu2016parameter,schroeder2017divergence}.
A relaxed $H(\divergence)$-conformity results in only discretely divergence free approximations and thus leads to a lack of pressure robustness. For this we present a simple reconstruction and provide numerical examples revealing the benefits of the modified Stokes discretization.

\noindent We note that we only treat an $h$-version error analysis here. The $hp$-version error analysis of the method will be treated in a forthcoming paper. Furthermore we want to refer to section 1.2 in \cite{LS_CMAME_2016} for a more detailed summery of works considering a discontinuous Galerkin approach.

\paragraph{Main contributions and structure of the paper}
In section \ref{hdgnse:stokes} we introduce new finite element methods. We start with a \emph{basic} method using an HDG formulation with relaxed $H(\divergence)$-conformity.
Further, we introduce a reconstruction operator to restore solenoidal velocity fields. Based on this operator we then define a \emph{modified} pressure robust discretization.
The relaxed $H(\divergence)$-conforming finite element space and the use of the reconstruction operator represent the first main contribution of this manuscript. Together, both allow to improve the computational efficiency of the underlying method while recovering the desirable properties of corresponding $H(\divergence)$-conforming formulations.
To the best of our knowledge, the method is the first HDG method for the Stokes equations that provides solenoidal solutions, pressure robustness and superconvergence properties simultaneously. 

The second main contribution of this paper is the presentation and characterization of a basis for the
(relaxed) $H(\divergence)$-conforming finite element space. We devote section~\ref{hdgnse:fesdetails} solely to aspects of these spaces, their bases and a reconstruction operator to map relaxed $H(\divergence)$-conforming functions to $H(\divergence)$-conforming functions.
The construction of the basis is non-standard and the efficient realization of the reconstruction operator heavily relies on the choice of basis functions. 
This section is the most technical part of the paper. We note however that the remaing part of the work can be understood without it.
A thorough a priori error analysis of the proposed methods, which represents the third main contribution of this paper, is carried out in section~\ref{sec:errorana}.
We conclude our work with numerical examples in section \ref{sec:numex} which confirm the theoretical findings.
\section{Relaxed $H(\divergence)$-conforming HDG methods for the Stokes problem} \label{hdgnse:stokes}
\subsection{ Preliminaries and Notation} \label{hdgnse:stokes:prelim}
We begin by introducing some preliminary notation and assumptions.
Let $\mesh$ be a shape-regular triangulation of an open bounded domain $\Omega$ in $\mathbb{R}^d$ with a Lipschitz boundary $\Gamma$. By $h$ we denote a characteristic mesh size. We note that $h$ can be understood as a local quantity, i.e. it can be different in different parts of the mesh due to a change in the local mesh size. The local character of $h$ will, however, not be reflected in the notation. 
The element interfaces and element boundaries coinciding with the domain boundary are denoted as \emph{facets}. The set of those facets $F$ is denoted by $\facets$ and there holds $\bigcup_{T\in \mesh} \partial T = \bigcup_{F\in \facets} F $. We separate facets at the domain boundary, exterior facets and interior facets by the sets $\facetsext$, $\facetsint$. 
For sufficiently smooth quantities we denote by
$\jumpDG{ \cdot }$ and $\averageDG{ \cdot }$ the usual jump and averaging operators across the facet $F \in \facetsint$. For $F \in \facetsext$ we define $\jumpDG{ \cdot }$ and $\averageDG{ \cdot }$ as the identity.

We restrict to the case where the mesh consists of straight simplicial elements $T$.
Further, in the analysis we consider only the case of homogeneous Dirichlet boundary conditions to simplify the presentation. By $\mathcal{P}^m(F)$ and $\mathcal{P}^m(T)$ we denote the space of polynomials up to degree $m$ on a facet $F\in\facets$ and an element $T \in \mesh$, respectively.
By $H^m(\Omega)$ we denote the usual Sobolev space on $\Omega$, whereas $H^m(\mesh)$ denotes its broken version on the mesh, $H^m(\mesh) := \{v \in L^2(\Omega) : v|_T \in H^m(T) ~ \forall ~ T \in \mesh \}$. At various parts we use the notation $( \cdot ,\cdot)_\omega$ as the $L^2$ product on a given domain $\omega$. 

In the discretization we introduce element unknowns which are supported on elements (in the volume) and different unknowns which are supported only on the skeleton of facets.
We indicate this relation with a subscript $\mathcal{F}$ for unknowns supported on the skeleton of facets and a subscript $\mathcal{T}$ for unknowns that are also supported on volume elements.
For the discretization of the velocity we use both type of functions and denote the compositions of volume and facet functions by $u_h = (\HdivVar,\FacetVar)$.
For restrictions of a function $\HdivVar$ ($\FacetVar$) to individual elements $T\in\mesh$ (facets $F\in\facets$) we use the subscripts $T$ ($F$), $\HdivVarT = \HdivVar|_T$($\FacetVarF = \FacetVar|_F$).

At several occasions we further distinguish tangential and normal directions of vector-valued functions. We therefore introduce the notation with a superscript $t$ to denote the tangential projection on a facet, $v^t = v - (v \cdot n)\cdot n \in \mathbb{R}^d$, where $n$ is the normal vector to a facet.
The index $k$ which describes the polynomial degree of the finite element approximation at many places through out the paper is an arbitrary but fixed positive integer number.

For the analysis we use the symbol $\partial_x$ for the partial derivative with respect to $x$ and in a similar manner we use the sub indices $(u)_1$ as symbol for the first component of a vectorial function $\VelVarEx$.
At several occasions we use the notation $a \lesssim b$  for $a,b \in \rr$ to express $a \leq c \ \! b$ for a constant $c$ that is independent of $h$.

\medskip

In this section we consider a new Hybrid DG discretization of the Stokes problem in velocity-pressure formulation.
The well-posed weak formulation of the problem is: Find $(\VelVarEx,\PressureVarEx) \in [H_0^1(\Omega)]^d \times L_0^2(\Omega)$, s.t.
\begin{equation}
  \left\{
    \begin{array}{r l l l}
      \displaystyle \int_{\Omega} \nu  {\nabla} \VelVarEx : {\nabla} \VelTestEx \, dx
      & \displaystyle  + \int_{\Omega} \divergence(\VelTestEx) \PressureVarEx \, dx
      & \displaystyle  = \int_{\Omega} {\ForceVar} \VelTestEx \ dx \quad
      & \displaystyle \mbox{ for all } \VelTestEx \in [H^1_0(\Omega)]^d, \\
      \displaystyle \int_{\Omega} \divergence(\VelVarEx) \PressureTestEx \, dx
      & \displaystyle 
      & \displaystyle = 0 \quad
      & \displaystyle \mbox{ for all } \PressureTestEx \in L^2_0(\Omega).
    \end{array} \right.
  \label{eq:stokes}
\end{equation}
Here, we have $L^2_0(\Omega) := \{ q \in L^2(\Omega) : \int_{\Omega} q\, dx = 0 \}$.
In the discretization we take special care about the treatment of the incompressibility condition which is closely related to the choice of finite element spaces, which will be introduced in the next subsection. 
\subsection{Finite element spaces} \label{hdgnse:stokes:fes}
In this subsection we introduce the finite element spaces that we require for our discretization. This involves several steps. 
First, we introduce finite element spaces for the velocity and the pressure as they have been used in DG discretizations for Stokes and Navier-Stokes problems in (among others) \cite{cockburn2007note,guzman2016h}. In a second step, we introduce facet unknowns for the velocity that facilitate a more efficient treatment of arising linear systems in the spirit of Hybrid DG methods, cf. \cite{lehrenfeld2010hybrid,cockburn2011analysis,egger2012hp,LS_CMAME_2016}.
We then introduce a modified velocity space which represents the essential novelty of our discretization.
Compared to the previously introduced space this modified velocity space is more involved. Details on the construction and realization of this space are separately treated in the subsequent subsection \ref{hdgnse:fesdetails}.
We also discuss an alternative approach which allows to circumvent the construction of this space in
\ifarxiv%
Section \ref{hdgnse:stokes:fesalternative} of the appendix.
\else%
\cite[Section A.1]{arXiv}.
\fi%

\subsubsection{$H(\divergence,\Omega)$-conforming finite elements: the \emph{$BDM_k$} space}
Although the velocity solution of the Stokes problem will typically be at least $H^1(\Omega)$-regular, we do not consider $H^1(\Omega)$-conforming finite elements. Instead, we base our discretization on (only) $H(\divergence,\Omega)$-conforming, i.e. normal-continuous, finite elements. These are very well-suited for a stable discretization of the incompressibility constraint \cite{lehrenfeld2010hybrid,cockburn2007note,LS_CMAME_2016}. Note that the non-conformity with respect to $H^1$ will be dealt with using Discontinuous Galerkin formulations in the discretization.
We recall the definition of $H(\divergence, \Omega)$:
\begin{align*}
  H(\divergence,\Omega)
  & := \{ v \in [L^2(\Omega)]^d :
    \divergence(v) \in L^2(\Omega) \}.
\end{align*}
  We use piecewise (vector-valued) polynomials for the approximation of the velocity, so that on each element the functions are automatically in $H(\divergence,T),~T\in\mesh$. For global conformity, continuity of the normal component is necessary. We consider the well-known \emph{$BDM_k$} space:
\begin{align}
  \HdivSpace
  & :=  \{ \HdivVar  \in \prod_{T\in\mesh} [\mathcal{P}^k(T)]^d : \jumpDG{ \HdivVar\! \cdot\! n } = 0,\ \forall F \in \facets \} \subset H(\divergence,\Omega).
\end{align}

\subsubsection{Tangential facet unknowns}
The space $\HdivSpace$ is not $H^1$-conforming so that tangential continuity has to be imposed weakly through a DG formulation for the viscosity terms, for instance as in \cite{cockburn2007note}.
However, the incorporation of tangential continuity in usual DG schemes leads to a huge amount of couplings between neighboring elements which increases the costs for solving the linear systems.
To approach this problem, we decouple element unknowns by introducing additional unknowns on the facets through which tangential continuity is implemented weakly. We note that such a mechanism is not necessary for the normal direction as normal-continuity is implemented (strongly) in the space $\HdivSpace$. We introduce the space for the facet unknowns
\begin{equation}
  \FacetSpace := \{ \FacetVar \in \prod_{F \in \facets} [\mathcal{P}^{k-1}(F)]^d : \FacetVarF \cdot  n = 0,\ \forall F \in \facetsint\},
\end{equation}
which can be seen as an approximation to the tangential trace of the velocity on the facets.
Note that we only consider polynomials up to degree $k-1$ in $\FacetSpace$ whereas we have order $k$ polynomials in $\HdivSpace$. Further, functions in $\FacetSpace$ have normal component zero.


\subsubsection{The pressure space}
For the pressure, the appropriate finite element space to the velocity space $\HdivSpace$ is the space of piecewise polynomials which are discontinuous and of one degree less:
\begin{equation} \label{eq:pressurespace}
  \PressureSpace := \prod_{T\in\mesh} \mathcal{P}^{k-1}(T).
\end{equation}
This velocity-pressure pair $\HdivSpace$/$\PressureSpace$ fulfills 
$\divergence ( \HdivSpace ) = \PressureSpace$ and we hence have that if a velocity $\HdivVar \in \HdivSpace$ is weakly incompressible, it is also strongly incompressible:
\begin{equation} \label{eq:stronginc}
  ( \divergence( \HdivVar) , \PressureTest )_\Omega = 0 \ \forall \PressureTest \in \PressureSpace \quad \Leftrightarrow \quad \divergence(\HdivVar)=0 \text{ in } \Omega.
\end{equation}
Hence, this choice of the velocity-pressure pair easily results in a pointwise divergence free solution. This property is crucial to obtain energy-stable schemes for Navier-Stokes discretizations, cf. \cite{cockburn2005locally}. We further note, that this finite element pair further has the remarkable property of providing an LBB-constant that is robust in the mesh size $h$ and the polynomial degree $k$, cf. Lemma \ref{lem:lbb} and \cite{LedererSchoeberl2016}.

\subsubsection{A modified $BDM_k$ space}
With the velocity unknowns in $\HdivSpace$ and $\FacetSpace$ we have polynomials up to degree $k$ associated to the normal direction of a facet, but only polynomials up to degree $k-1$ for the tangential direction of a facet. As these unknowns are the globally coupled unknowns, they essentially decide on the computational costs associated with solving linear systems. The question arises if we can reduce the polynomial degree for the normal direction also to degree $k-1$. We can achieve this only by relaxing the normal continuity of $\HdivSpace$.
To this end, we introduce a modified Finite Element space with ``relaxed $H(\divergence)$-conformity'', 
\begin{equation} \label{eq:weaklyhdivconf}
  \HdivSpaceHODC := \{ \HdivVar  \in \prod_{T\in\mesh} [\mathcal{P}^k(T)]^d: \facetproj^{k-1} \jumpDG{ \HdivVar\! \cdot\! n } = 0, \ \forall \ F \in \facets \} \not \subset H(\divergence,\Omega),
\end{equation}
where $\facetproj^{k-1}: [L^2(F)]^d \rightarrow [\mathcal{P}^{k-1}(F)]^d$ is the $L^2(F)$ projection into $[\mathcal{P}^{k-1}(F)]^d$:
\begin{equation} \label{eq:facetproj}
 ( \facetproj^{k-1} \ \! {w} ,  {\DiscTest})_F = ({w}, {\DiscTest})_F \quad \forall \ {\DiscTest} \in [\mathcal{P}^{k-1}(F)]^d.
\end{equation}
Details on the constructions of the finite element space $\HdivSpaceHODC$ are given below, in section \ref{hdgnse:fesdetails}.
Functions in $\HdivSpaceHODC$ are only ``almost normal-continuous'', but can be normal-discontinuous in the highest orders. The relaxation reduces the number of unknowns that are shared by neighboring elements which improves the sparsity pattern of the linear systems.
However, this modification comes at a price. 
We obviously have $\HdivSpace \subset \HdivSpaceHODC$ and $\HdivSpace|_T = \HdivSpaceHODC|_T,~\forall~ T \in \mesh$, but the pair $\HdivSpaceHODC$/$\PressureSpace$ only has the local property
$
\divergence ( \HdivSpaceHODC|_T ) = \divergence ( \HdivSpace|_T ) = \PressureSpace|_T~ \forall~ T \in \mesh$. 
If a velocity $\HdivVar \in \HdivSpaceHODC$ is weakly incompressible (in the sense of the left hand side of \eqref{eq:stronginc}), we only have
\begin{equation} \label{eq:stronginc2}
  \divergence(\HdivVarT)=0 \text{ in } T \in \mesh \text{ and } \Pi^{k-1} \jumpDG{ u_T \cdot n } = 0,~\forall~T\in\mesh.
\end{equation}
In order to obtain mass conservative velocity fields we are missing normal continuity in the higher order moments, $(\operatorname{id} - \Pi^{k-1}) \jumpDG{ u_T \cdot n } = 0$.

For the discretization of the velocity field we use the composite space 
\begin{equation}
  \VelSpace := \HdivSpaceHODC \times \FacetSpace.
\end{equation}
and define the tangential jump operator $\facetjump{\VelVar^t } = \HdivVarT^t|_F - u_F,~ F \in \facetsint,~ T \in \mesh$. We notice that the jump $\facetjump{ \VelVar^t }$ is element-sided that means that it can take different values for different sides of the same facet. Further, note that for $F \in \facetsext$ we have $u_F=0$ so that $\facetjump{ \VelVar^t } = \HdivVar^t|_F$.

A suitable discrete norm on $\VelSpace$ which mimics the $H^1(\Omega)$ norm and a suitable norm for the velocity pressure space $\VelSpace \times \PressureSpace$ are
\begin{equation} \label{eq:discretenorms}
  \brokenHnorm{\VelVar}^2
   \!:= \!\!\! \sumoverallelements
    \!\left\{ \! \Vert {\nabla} \HdivVarT \Vert_T^2  + \frac{1}{h} \Vert \facetproj^{k-1} \facetjump{ \VelVar^t  } \Vert_{\partial T}^2 \right\}\!,~
  \brokenHnorm{(\VelVar,\PressureVar)}
  \!:= \! \sqrt{\nu} \brokenHnorm{\VelVar}\! +\! \frac{1}{\sqrt{\nu}} \Vert \PressureVar \Vert_{L^2(\Omega)}.
\end{equation}

\subsection{Relaxed $H(\divergence)$-conforming HDG formulation} \label{sec:hdgstokes}
The basis of our Hybrid DG formulation is the formulation in \cite{LS_CMAME_2016} which is based on the spaces $\HdivSpace$, $\PressureSpace$ and $\FacetSpace$.
In contrast to that formulation, we now replace $\HdivSpace$ with $\HdivSpaceHODC$.
This will lead to an order-optimal \emph{basic} method with a reduced number of globally coupled unknowns.
Due to the relaxation of the normal-continuity velocity solutions will (in contrast to the discretization in \cite{LS_CMAME_2016} with $\HdivSpace$) be neither exactly solenoidal nor pressure robust, i.e. the error in the velocity solution will depend on the approximation of the pressure.
However, both drawbacks can be dealt with using a simple and computationally cheap reconstruction operator. 
In section \ref{hdgnse:stokes:reconstruction}, we address the issue of restoring normal-continuity strongly (after the solution of linear systems) by a reconstruction operator.
In section \ref{hdgnse:stokes:pressurerobust}, we discuss how to also restore pressure robustness. 

The \emph{basic} discretization is as follows:
Find $\VelVar=(\HdivVar, \FacetVar) \in \VelSpace$ and $\PressureVar \in \PressureSpace$, s.t.
\begin{equation}\label{eq:discstokes}
  \displaystyle
  \left\{
  \begin{array}{crcrl}
    \bilinearform{A}(\VelVar,{\VelTest})
    & + \ \ \bilinearform{B}({\VelTest},\PressureVar)
    & =
    & \ForceVar ( \VelTest )
    & \forall \ \VelTest \in \VelSpace,\\
    & \bilinearform{B}(\VelVar,\PressureTest)
    & =
    & 0
    & \forall \ \PressureTest \in \PressureSpace,
  \end{array}
      \right.
      \tag{B}
\end{equation}
with the bilinear forms corresponding to viscosity ($\bilinearform{A}$), pressure and incompressibility ($\bilinearform{B}$) defined in the following.
For the viscosity we introduce the bilinear form $\bilinearform{A}$ for $\VelVar, \VelTest \in \VelSpace$ as 
\begin{align} \label{eq:blfA}
    \bilinearform{A}(\VelVar,\VelTest) :=
    & \displaystyle \sumoverallelements \int_{T} \nu {\nabla} {\HdivVarT} \! : \! {\nabla} {\HdivTestT} \ d {x} - \int_{\partial T} \nu \frac{\partial {\HdivVarT}}{\partial {\normal} }  \facetproj^{k-1} \facetjump{ \VelTest^t } \ d {s} \\
    & \displaystyle- \int_{\partial T} \nu \frac{\partial {\HdivTestT}}{\partial {\normal} } \facetproj^{k-1} \facetjump{ \VelVar^t } \ d {s}
      + \int_{\partial T} \nu \frac{\stab k^2}{h} \facetproj^{k-1} \facetjump{ \VelVar^t }  \facetproj^{k-1} \facetjump{ \VelTest^t } \ d {s}, \nonumber
\end{align}
where $\stab$ is chosen, such that the bilinearform is coercive w.r.t.
$ \brokenHnorm{\cdot}$ on $\VelSpace$, see Lemma \ref{lem:coercivity} below.
The $L^2(F)$-projection $\facetproj^{k-1}$, c.f. \eqref{eq:facetproj}, realizes a reduced stabilization \cite[Section 2.2.1]{LS_CMAME_2016}.
If we replace $\facetproj^{k-1}$ with $\operatorname{id}$ in \eqref{eq:blfA} we obtain the usual hybridized interior penalty formulation of the viscosity as used in \cite{lehrenfeld2010hybrid,rhebergenwells2016}.
In the literature of DG methods many alternatives to the interior penalty method are known. For many of these alternatives there is a corresponding HDG version which is also applicable, we however restrict to the interior penalty method for ease of presentation.
Implementational aspects of the projection operator $\facetproj^{k-1}$ have been discussed in \cite[Section 2.2.2]{LS_CMAME_2016}.

The bilinear form for the pressure part and the incompressibility constraint is
\begin{equation}\label{eq:blfB}
  \bilinearform{B}({\VelVar},\PressureVar) := \sumoverallelements - \int_{T} \PressureVar 
  \divergence ({\HdivVarT}) \ d {x} \quad \text{ for } \VelVar \in \VelSpace, \ \PressureVar \in \PressureSpace.
\end{equation}

We denote the discretization in \eqref{eq:discstokes} as our \emph{basic} discretization.
Note that despite the naming, the method is \emph{hybrid} only with respect to the tangential velocity trace. In
\ifarxiv%
Section \ref{hdgnse:stokes:fesalternative} of the appendix
\else%
\cite[Section A.1]{arXiv}
\fi%
we discuss an equivalent \emph{hybridized} formulation based on scalar finite element spaces.
In the next subsections we introduce modifications related to the relaxed $H(\divergence)$-conformity.

\subsection{A reconstruction operator for strong $H(\divergence)$-conformity}\label{hdgnse:stokes:reconstruction}
Velocity solutions to \eqref{eq:discstokes} are in general not solenoidal as $(\operatorname{id} - \facetproj^{k-1}) \jumpDG{ {\HdivVar} \cdot n }$ can be non-zero.
In Navier-Stokes or other coupled problems where the velocity solution serves as an advective velocity it is benefitial to have an exact solenoidal field.
Similarly to the approach in \cite{cockburn2005locally} we propose to use a reconstruction operator for the velocity solution which restores $H(\divergence)$-conformity.
We denote such a reconstruction operator as $\ReconHdiv: \HdivSpaceHODC \to \HdivSpace$
and define the canonical extension to $\HdivSpace \times \FacetSpace$ as
\begin{equation}
  \Recon: \quad \HdivSpaceHODC \times \FacetSpace \to  \HdivSpace \times \FacetSpace, \qquad \Recon \VelVar := (\ReconHdiv \HdivVar, \FacetVar ). 
\end{equation}
We make the following assumptions on this reconstuction operator:
\begin{assumption} \label{ass:recon}
  We assume that the reconstruction operators $\ReconHdiv$ and $\Recon$, respectively, fulfill the following conditions:
  \begin{subequations} \label{eq:ass}
    \begin{align}
      \ReconHdiv \HdivTest
      & \in \HdivSpace
      && \text{for all } \HdivTest \in \HdivSpaceHODC, \label{eq:ass1}\\
      (\ReconHdiv \HdivTest \!\cdot\! n, \varphi)_{F}
      & =
      (\HdivTest\! \cdot\! n, \varphi)_{F}
      && \text{for all } \varphi \in \mathcal{P}^{k-1}(F), 
         \HdivTest \in \HdivSpaceHODC,~ F \in \facets,  \label{eq:ass2a}\\      
      (\ReconHdiv \HdivTest, \varphi)_{T}
      & =
      (\HdivTestT, \varphi)_{T}
      && \text{for all } \varphi \in [\mathcal{P}^{k-2}(T)]^d, 
         \HdivTest \in \HdivSpaceHODC,~ T \in \mesh,  \label{eq:ass2b}\\
    \brokenHnorm{\Recon \VelTest} & \lesssim \brokenHnorm{\VelTest} && \text{for all } \VelTest \in \VelSpace. \label{eq:ass3}
    \end{align}
  \end{subequations}
\end{assumption}
Here, \eqref{eq:ass1} ensures $H(\divergence)$-conformity, \eqref{eq:ass2a} ensures that the modes up to order $k-1$ of the normal component are not changed by the operator $\ReconHdiv$, \eqref{eq:ass2b} describes a constistency property in $L^2(\Omega)$ and \eqref{eq:ass3} ensures stability in the discrete energy norm. 
A reconstruction operator with these properties maps weakly divergence free velocities (in the sense of \eqref{eq:discstokes}) onto pointwise divergence free velocities, cf. Lemma \ref{lem:recon} below.

One possible choice is a (Discontinuous Galerkin) generalization of the classical BDM interpolation \cite[Proposition 2.3.2]{brezzi2012mixed}, $\ReconHdiv^{BDM} : [H^1(\mesh)]^d \rightarrow \HdivSpace$. It has also been used in \cite{guzman2016h}. The interpolation is defined element-by-element for $\HdivVar \in [H^1(\mesh)]^d$ by
\begin{subequations} \label{eq:bdmprops}
  \begin{align}
    ( \ReconHdiv^{BDM}\HdivVar\!\cdot\! n, \varphi)_F
    & = ( \averageDG{\HdivVar}\!  \cdot \!n, \varphi)_F
    \hspace*{-0.25cm}&\hspace*{0cm}& \text{ for all } \varphi \in \mathcal{P}^k(F), \ F \subset \partial T, \label{bdm1}\\
    ( \ReconHdiv^{BDM}\HdivVar,\varphi)_T
    & = ( \HdivVarT, \varphi)_T
      \hspace*{-0.25cm}&\hspace*{0cm}& \text{ for all } \varphi \in \mathcal{N}^{k-2}(T), \label{bdm2}
  \end{align}
\end{subequations}
with $\mathcal{N}^{k-2}(T) := [\mathcal{P}^{k-2}(T)]^d + [\mathcal{P}^{k-2}(T)]^d \times x$.
We note that the definition of the DG-version BDM interpolation is slightly different from the one used in \cite{cockburn2005locally}. We note that this reconstruction operator satisfies all the previously mentioned assumptions. The assumptions \eqref{eq:ass1}, \eqref{eq:ass2a} and \eqref{eq:ass2b} are satisfied by construction and stability -- in the sense of \eqref{eq:ass3} -- has also been shown already in \cite{cockburn2005locally}.
The reconstruction operator in \cite{cockburn2005locally} is designed for a fully discontinuous velocity space.

Based on the finite element basis that we discuss in section \ref{hdgnse:fesdetails} we propose a simple and efficient realization of a similar, but different, interpolation operator (see section \ref{hdgnse:stokes:reconstruction:average})  which also fulfills assumption \ref{ass:recon}. 

\subsection{Pressure robust relaxed $H(\divergence)$-conforming HDG formulation}\label{hdgnse:stokes:pressurerobust}
It is well known that irrotational parts (in the sense of a continuous Helmholtz decomposition) of an exterior force $\ForceVar$ are $L^2$-orthogonal on exactly divergence free velocities $\VelTestEx \in [H^1_0(\Omega)]^d$. Nevertheless this orthogonality may not hold true in the discrete case leading to a velocity error estimate which depends on the pressure approximation \cite{linke2014role}. Recent works \cite{blms:2015, 2016arXiv160903701L,Linke:2012,lmt:2016} considering different velocity and pressure spaces have shown that a modification of the right hand side allows to obtain pressure-robust, thus pressure independent, error estimates. This is achieved by mapping weakly divergence free test functions to exactly divergence free functions resulting in a restored $L^2$ orthogonality with respect to the irrotational parts of $\ForceVar$. Note that $H(\divergence)$-conforming methods as for example \cite{LS_CMAME_2016, cockburn2007note} provide exactly divergence free velocities and thus do not suffer from the described problems. Due to the modification \eqref{eq:weaklyhdivconf} functions $\VelVar \in \VelSpace $ are only weakly divergence free, see equation \eqref{eq:stronginc2}. The reconstruction operator introduced in the Section \ref{hdgnse:stokes:reconstruction} allows to obtain a pressure robust discretization also for our new method.
The discretization then takes the form:
Find $\VelVar \in \VelSpace$ and $\PressureVar \in \PressureSpace$, s.t.
\begin{equation}\label{eq:probuststokes}
  \displaystyle
  \left\{
  \begin{array}{crcrl}
    \bilinearform{A}(\VelVar,{\VelTest})
    & + \ \ \bilinearform{B}({\VelTest},\PressureVar)
    & =
    & \ForceVar( \Recon \VelTest)
    & \forall \ \VelTest \in \VelSpace,\\
    & \bilinearform{B}(\VelVar,\PressureTest)
    & =
    & 0
    & \forall \ \PressureTest \in \PressureSpace.
  \end{array}
      \right.
      \tag{PR}
\end{equation}
Solutions $\VelVar \in \VelSpace$ to \eqref{eq:probuststokes} are not exactly divergence free but provide a pressure independent velocity error estimate. This involves a Strang type consistency estimate and is proven in section \ref{analysis:probust}. Furthermore, solutions to \eqref{eq:probuststokes} can be post-processed with a subsequent application of $\Recon$ to obtain an exactly divergence free solution.

\section{On the construction of the finite element spaces} \label{hdgnse:fesdetails}
In this section we address the construction of the finite element spaces $\HdivSpaceHODC$, $\FacetSpace$, an efficient reconstruction operator $\ReconHdiv$ and the realization of the projected jumps operator $\facetproj^{k-1}$.

In the subsections \ref{section:localspaces} and \ref{sec:proplocalbasis} we introduce the finite element basis functions on a reference element and explain how these are composed to a global finite element space $\HdivSpaceHODC$ in subsection \ref{section:globalspaces}. Based on these preparations we can introduce an efficient reconstruction operator $\ReconHdiv$ which meets the requirements of Assumption \ref{ass:recon} in subsection \ref{hdgnse:stokes:reconstruction:average}.

\subsection{Construction of a local $H(\divergence)$-conforming FE Space} \label{section:localspaces}
In this section we define shape functions to construct a basis of a local space $\HdivSpaceref$ on the reference triangle $\widehat{T}$ given as the convex hull of the vertices $\mathcal{V} = \{V_i\}_{i=1}^3 :=\{(-1,0), (1,0), (0,1) \}$ and $\mathcal{V}= \{V_i\}_{i=1}^4 := \{ (-1,0,0), (1,0,0), (0,1,0), (0,0,1) \}$ for two and three dimensions respectively.
Using a proper transformation we can then construct the global spaces $\HdivSpace$ and $\HdivSpaceHODC$, see section \ref{section:globalspaces}. For the construction of the local space $\HdivSpaceref$ we refer to \cite{Beuchler2012} where the basic concepts presented in \cite{zaglmayr2006high} are combined with some adaptations leading to a sparsity optimized high order basis. The novel choice of the element basis functions is such that their normal component form a hierarchical $L^2$ orthogonal basis on faces, see Lemma \ref{lem:normalltwoorthogonal}. The idea in \cite{zaglmayr2006high} is to construct conforming finite element spaces  for $H^1, H(\divergence), H(\curl)$ and $L^2$ which fit in the setting of the exact de Rham Complex.
This construction allows to split the space $\HdivSpace \subset H(\divergence)$ into divergence free basis functions given by the $\curl$ of basis functions in $\HcurlSpace \subset H(\curl)$ and functions with a non zero divergence.
On the reference triangle this leads to the following decomposition 
\begin{equation}\label{eq:decomposition}
  \HdivSpaceref = RT_0 \oplus \HdivSpaceref^{\text{F}} \oplus \HdivSpaceref^{\text{divfree,T}} \oplus \HdivSpaceref^{\text{div,T}}, 
\end{equation}
with the lowest order Raviart-Thomas subspace $RT_0$, see \cite{Nedelec1986}, the subspace of higher order divergence free facet functions $\HdivSpaceref^{\text{F}}$
the subspace of cell based divergence free basis functions $\HdivSpaceref^{\text{divfree,T}}$ and the subspace of cell based functions that have a non zero divergence $\HdivSpaceref^{\text{div,T}}$.

Now let $\jacobi_n^{(\alpha,\beta)}$ be the $n$-th Jacobi polynomial and $\intjacobi_n^{(\alpha,\beta)}$ the $n$-th integrated Jacobi polynomial, cf. \cite{Abramowitz,andrews,MR2221053},
\begin{align*}
 \jacobi_n^{(\alpha,\beta)}(x)\! &:=\frac{1}{2^n n! (1-x)^\alpha(1+x)^\beta} \frac{\mathrm{d}^n}{\mathrm{d} x^n} \!\left( (1\!-\! x)^\alpha(1\!+\!x)^\beta(x^2\!\!-\!1)^n\right)\!, && n \in \mathbb{N}_0, \alpha,\beta \!>\! -1,
 \\
  \intjacobi_n^{(\alpha,\beta)}(x) &:= \int_{-1}^x \jacobi_{n-1}^{(\alpha,\beta)}(\xi) \ d {\xi}, && n\ge 1, \intjacobi_0^{\alpha}(x) = 1.
\end{align*}
In our case we use $\beta = 0$ so we skip to the simpler notation $\jacobi_n^{(\alpha,0)}(x) = \jacobi_n^{\alpha}(x)$ and $\intjacobi_n^{(\alpha,0)}(x) = \intjacobi_n^{\alpha}(x)$. Note that there holds the following orthogonality property.
\begin{align} 
  \int_{-1}^1 (1-x)^\alpha \jacobi_j^\alpha (x) \jacobi_l^\alpha (x) \ d {x} & = c_j^\alpha \delta_{jl} \hspace*{-0.2cm}&& \text{with} \quad  c_j^\alpha = \frac{2^{\alpha+1}}{2j+ \alpha + 1}, \label{eq:jacobiortho} \\
  \int_{-1}^1 (1-x)^\alpha \intjacobi_j^\alpha (x) \intjacobi_l^\alpha (x) \ d {x} & = 0 \hspace*{-0.2cm}&& \text{for} \quad |j-l| > 2. \label{eq:intjacobiortho}
\end{align}
Furthermore we define $\lambda_i \in \mathcal{P}^1(\widehat{T})$ as the barycentric coordinates uniquely determined by $\lambda_i(V_j) = \delta_{ij}$.
\subsubsection{Basis for two dimensions} \label{2dbasis}
For a fixed order $k$ we use the same construction of the basis as presented in chapter 3 in \cite{Beuchler2012}.  Let $[f_1,f_2]$ be the edge running from vertex $V_{f_1}$ and $V_{f_2}$ and define $u_i := \intjacobi_i^{0} \left( \frac{\lambda_{2} -\lambda_{1}}{\lambda_{2} + \lambda_{1}}\right) (\lambda_{2} + \lambda_{1})^i$ and $v_{ij} := \intjacobi_j^{2i-1}(2\lambda_3-1)$. Then we have: \newline
\begin{itemize}
\item The lowest order Raviart Thomas basis functions $\phi_0^{[e_1,e_2]} := \curl(\lambda_{e_1}) \lambda_{e_2} - \lambda_{e_1}\curl(\lambda_{e_2})$ such that \vspace*{-0.15cm}
  \begin{align*}
    {RT}_0 &= \hull (\{\phi_0^{[1,2]}, \phi_0^{[2,3]},\phi_0^{[3,1]}\}).
  \end{align*} ~ \\[-6ex]
\item High order edge based basis functions defined by \\$\phi_i^{[f_1,f_2]} := \curl \left( \intjacobi_{i+1}^{0}\left(\frac{\lambda_{f_2} -\lambda_{f_1}}{\lambda_{f_2} + \lambda_{f_1}}\right) (\lambda_{f_2} +\lambda_{f_1})^{i+1} \right)$ and
  ${\Phi}^{[f_1,f_1]} := \{ \phi_i^{[e_1,e_2]} \}$ for $1 \le i \le k$, such that \vspace*{-0.15cm}
  \begin{align*}
  \HdivSpaceref^{\text{F}} &= \hull( {\Phi}^{[1,2]}) \oplus \hull( {\Phi}^{[2,3]} ) \oplus \hull( {\Phi}^{[3,1]} ).
  \end{align*} ~ \\[-6ex]
\item Divergence free cell based basis functions $\phi_{ij}^{(a)}:=\curl \left( u_i v_{ij} \right)$ with ${\Phi}^{(a)} :=\{ \phi_{ij}^{(a)} \}$ for $i \ge 2, j \ge 1, i+j \le k+1$, such that \vspace*{-0.15cm}
  \begin{align*}
    \HdivSpaceref^{\text{divfree,T}} &= \hull( {\Phi}^{(a)}).
  \end{align*} ~ \\[-6ex]
\item Cell based basis functions with a non zero divergence $\phi_{1l}^{(b)}:=2 \phi_0^{[1,2]} \intjacobi_l^3(2\lambda3-1)$ and $\phi_{ij}^{(c)}:= \curl (u_i) v_{ij}$ with ${\Phi}^{(b)} := \{ \phi_{1l}^{(b)} \}$ for $1\le l \le k-1$ and ${\Phi}^{(c)} := \{ \phi_{ij}^{(c)} \}$ for $i \ge 2, j \ge 1, i+j \le k +1$, such that \vspace*{-0.15cm}
  \begin{align*}
    \HdivSpaceref^{\text{div,T}} &= \hull( {\Phi}^{(b)}) \oplus \hull( {\Phi}^{(c)}).
  \end{align*} ~ \\[-6ex]
\end{itemize}
\subsubsection{Basis for three dimensions} \label{3dbasis}
Let $k$ be a fixed order. For our local basis we nearly use the same construction as presented in chapter 4 in \cite{Beuchler2012}.  Let $F = [f_1,f_2, f_3]$ be the face defined as the convex hull of $V_{f_1},V_{f_2}$ and $V_{f_3}$ and define
\begin{align*}
  u_i &:= \intjacobi_i^{0} \left( \frac{\lambda_{2} -\lambda_{1}}{\lambda_{2} + \lambda_{1}}\right) (\lambda_{2} + \lambda_{1})^i, \quad  w_{ijk} := \intjacobi_j^{2i+2j-2} (2\lambda_4-1),  \\
 v_{ij} &:= \intjacobi_j^{2i-1}\left(\frac{2\lambda_3-(1-\lambda_4) }{1-\lambda_4} \right) (1-\lambda_4)^j,\\
  u^F_i &:= \intjacobi_i^{0} \left( \frac{\lambda_{f_2} -\lambda_{f_1}}{\lambda_{f_2} + \lambda_{f_1}}\right) (\lambda_{f_2} + \lambda_{f_1})^i, \quad   \mathcal{N}_0^{[f_1,f_2]} :=  \nabla \lambda_{f_1} \lambda_{f_2} - \lambda_{f_1} \nabla \lambda_{f_2}, \\
  v^F_{ij} &:= \intjacobi_j^{2i-1}\left(\frac{\lambda_{f_3}- \lambda_{f_2} - \lambda_{f_1}}{\lambda_{f_3} + \lambda_{f_2} + \lambda_{f_1}} \right) (\lambda_{f_3}+ \lambda_{f_2} + \lambda_{f_1})^j,
\end{align*}
where $\mathcal{N}_0^{[f_1,f_2]}$ is the lowest order Ned\'{e}l\'{e}c function for the edge $[f_1,f_2]$. Then we have:
\begin{itemize}
\item The lowest order Raviart Thomas basis functions $\phi_{00}^{[f_1,f_2,f_3]} := \lambda_{f_1} \nabla \lambda_{f_2} \times \nabla \lambda_{f_3} + \lambda_{f_2} \nabla \lambda_{f_3} \times \nabla \lambda_{f_1} + \lambda_{f_3} \nabla \lambda_{f_1} \times \nabla \lambda_{f_2}$, such that
  \begin{align*}
    {RT}_0 &= \hull ({\Phi}_0) \quad \text{with} \quad {\Phi}_0 := \{\phi_{00}^{[1,2,3]}, \phi_{00}^{[1,3,4]},\phi_{00}^{[1,2,4]},\phi_{00}^{[2,3,4]}\}
  \end{align*}
  \item High order face based basis functions $\phi_{0l}^{F} := \nabla \times (\mathcal{N}_0^{[f_1,f_2]} v_{2l}^F )$ and $\phi_{ij}^{F} := \nabla \times \left( \nabla u_{i+1}^Fv_{(i+1)(j+1)}^F \right) = -\nabla u_{i+1}^F \times \nabla v_{(i+1)(j+1)}^F $  with  ${\Phi}^{F} := \{ \phi_{0l}^{F} \} \cup \{\phi_{ij}^{F} \}$ for $1 \le l \le k$ and $i \ge 1, i+j \le k$, such that
  \begin{align*} 
  \HdivSpaceref^{\text{F}} &= \hull( {\Phi}^{[1,2,3]})  \oplus \hull( {\Phi}^{[1,3,4]}) \oplus \hull( {\Phi}^{[1,2,4]}) \oplus \hull( {\Phi}^{[2,3,4]}).
  \end{align*}
\item divergence free cell based basis functions $\phi_{1jl}^{(a)}:=\nabla \times  (\mathcal{N}_0^{[1,2]} v_{2j} w_{2jl} )$ and $ \phi_{ijl}^{(b)}:=\nabla \times  \left( \nabla u_i v_{ij} w_{ijl} \right)$ and $\phi_{ijl}^{(c)}:=\nabla \times  \left( \nabla (u_i v_{ij}) w_{ijl} \right)$ with ${\Phi}^{(a)} :=\{ \phi_{1jl}^{(a)} \}$ for $j,l \ge 1, j+l \le k$, and ${\Phi}^{(b)} :=\{ \phi_{ijl}^{(b)} \}$ for $i \ge 2, j,l \ge 1, 1+j+l \le k+2$, and ${\Phi}^{(c)} :=\{ \phi_{ijl}^{(c)} \}$ for $i \ge 2$ and $j,l \ge 1, i+j+l \le k+2$, such that
  \begin{align*}
    \HdivSpaceref^{\text{divfree,T}} &= \hull( {\Phi}^{(a)}) \oplus \hull( {\Phi}^{(b)}) \oplus \hull( {\Phi}^{(c)}). 
  \end{align*}
\item Cell based basis functions with a non zero divergence $\phi_{10l}^{(d)}:= 4 \phi_0^{[1,2,3]} w_{21l}$, and $\phi_{1jl}^{(e)}:= 2 \mathcal{N}_0^{[1,2]}  \times \nabla w_{2jl} v_{2j}$, and $\phi_{ijl}^{(f)}:=  w_{ijl} \nabla u_i \times \nabla v_{ij}$ with ${\Phi}^{(d)} := \{ \phi_{10l}^{(d)} \}$ for $1 \le l \le k-1$ and ${\Phi}^{(e)} := \{ \phi_{1jl}^{(e)} \}$ for $j,l \ge 1, j+l \le k$, and ${\Phi}^{(f)} := \{ \phi_{ijl}^{(f)} \}$ for $i \ge 2, j,l \ge 1, i+j+l \le k+2$, such that
  \begin{align*} 
    \HdivSpaceref^{\text{div,T}} &= \hull( {\Phi}^{(d)}) \oplus \hull( {\Phi}^{(e)}) \oplus \hull( {\Phi}^{(f)}). 
  \end{align*}
\end{itemize}
\begin{remark}\label{rem:hodivfree}
Due to the decomposition \eqref{eq:decomposition} we can conclude that for an arbitrary function $v_h$ in $\HdivSpaceref$ with $\divergence({v_h})=0$ the coefficients corresponding to the basis functions of the subspace $\HdivSpaceref^{\text{div,T}}$, which have a non zero divergence, are equal to zero. We can use this for the approximation of the saddle point problem \eqref{eq:discstokes} keeping in mind the pressure variable can be interpreted as a Lagrangian multiplier to fulfill the incompressibility constraint $\divergence (\VelVar) = 0$. Where the lowest order piece wise constant pressure basis functions compensate the divergence of the Raviart Thomas basis functions of the velocity, the high order part of the pressure space is needed to handle the divergence of velocity basis functions in $\HdivSpaceref^{\text{div,T}}$. In order to reduce the size of the problem one can now simply remove the basis functions of $\HdivSpaceref^{\text{div,T}}$ and use the pressure space $\PressureSpace := \prod_{T\in\mesh} \mathcal{P}^{0}(T)$ instead. Note that this has no influence in the (optimal) error estimation of the velocity and furthermore one can use a simple element-wise post processing to obtain a high order approximation of the pressure.
\end{remark}
\begin{remark}
  The differences of the presented basis compared to the one given in \cite{Beuchler2012} are the high order face based basis functions $\phi^{F}_{0l}$ where we have chosen a different index and a proper scaling into the direction of the opposite vertex of the face function $v_{ij}^F$. Note that this has no influence on the linear in dependency, thus the presented set of basis functions is still a basis for $\HdivSpaceref$. Furthermore we still have $\divergence(\phi^{F}_{0l}) = 0$, thus we have the same sparsity pattern of the $\divergence\divergence$-stiffness matrix and even a better sparsity pattern of the mass matrix (see Lemma \ref{lem:facetortho})  as in \cite{Beuchler2012}.
\end{remark}
\subsection{Properties of the local basis} \label{sec:proplocalbasis}
Beside the sparsity properties of the stiffness and mass matrix (see \cite{Beuchler2012}) for the basis introduced in section \ref{2dbasis} and \ref{3dbasis} in this section, we proof properties that are important for proving Assumption \ref{ass:recon} for the reconstruction operator that we present below, in section \ref{hdgnse:stokes:reconstruction:average}.
\begin{lemma}[$L^2$-normal orthogonality] \label{lem:normalltwoorthogonal}
The basis presented in section~\ref{section:localspaces} has a $L^2$-orthogonal normal trace. 
In two dimensions there holds ($c\!=\!c(i,j)\!>\!0$)
\begin{subequations}
\begin{align}
                                                            (\phi^F_i \cdot n ,\phi^F_j \cdot n)_F &= c \delta_{ij}  \quad \forall~i,j=0,..,k \quad \forall F \subset \partial \widehat{T}.\label{eq:normalltwoorthogonaltwodsecond}
\end{align}
In three dimensions there holds ($c = c(i,j,l,m)>0$)
\begin{align}
 (\phi^F_{ij} \cdot n,\phi^F_{lm} \cdot n)_F &= c \delta_{il} \delta_{jm},~\forall i,j,l,m\leq k,~~ i\!+\!j,l\!+\!m\!\leq\! k,~~\forall F \subset \partial \widehat{T}. \label{eq:normalltwoorthogonalthreedsecond}
\end{align}
\end{subequations}
\end{lemma}
\vspace*{-0.25cm}
\begin{proof}
  We start with the two dimensional case and the lower edge $F = [1,2]$. For $i,j \ge 1$ the basis functions are given by  $\phi_i^{[1,2]} = \curl \left( \intjacobi_{i+1}^{0}\left(\frac{\lambda_{2} -\lambda_{1}}{\lambda_{2} + \lambda_{1}}\right) (\lambda_{2} +\lambda_{1})^{i+1} \right)$. In two dimensions the $\curl$ times the normal vector equals the tangential derivative, thus for $F$ the partial derivation with respect to $x$
  \begin{align*}
    (\phi^{[1,2]}_i \cdot n)|_F =  \partialsym{x} {\intjacobi_{i+1}}^{0}(x) = \jacobi^{0}_{i}(x),
  \end{align*} where we used that $\lambda_2+\lambda_1 = 1$ and $\lambda_2-\lambda_1 = x$ on $F$. Next note that for $i,j=0$ the normal component of the lowest order Raviart Thomas basis function $\phi^F_0$ is constant on $F$ thus equivalent to $\jacobi^{0}_0$. Using property \eqref{eq:jacobiortho} it follows
  \begin{align*}
    \int_{F} (\phi^{[1,2]}_i \cdot n) (\phi^{[1,2]}_j \cdot n) \ d {s} &= \int_{-1}^1\jacobi^{0}_{i}(x) \jacobi^{0}_{j}(x) \ d {x} = c_i \delta_{ij} \quad  \text{for} \quad 0 \le i,j \le k
  \end{align*}
  The other two edges follow analogue. For the proof in three dimensions let $F =[1,2,3]$ and start with $i \ge 1$.  The normal vector on $F$ is given by $n=(0,0,-1)$, so we have 
  \begin{align*}
    \phi_{ij}^{F}\big{|}_F \cdot n =&(-\nabla u_{i+1}^F \times \nabla v_{(i+1)(j+1)}^F ) \cdot n
    =\jacobi_{i}^{0} \Big( \frac{x}{1-y}\Big) (1-y)^{i} \jacobi_{j}^{2i+1}\left(2y-1\right)  2.
  \end{align*}
  Using a Duffy transformation from $D_1: (-1,1) \times (0,1) \rightarrow F$, with $D_1(\hat{x}, \hat{y}) = (\hat{x} (1-\hat{y}), \hat{y})$ we get
    \begin{align*}
      (\phi^{F}_{ij}& \!\cdot\! n, \phi^{F}_{lm} \!\cdot\! n)_F 
                = 4 \int_{-1}^1 \jacobi_i^{0} \left( \hat{x} \right)  \jacobi_l^{0} \left( \hat{x} \right) \! \! \ d {\hat{x}} \! \int_0^1  \jacobi_j^{2i+1}\left(2 \hat{y} -1\right)  \jacobi_m^{2l+1}\left(2 \hat{y} -1\right) (1-\hat{y})^{i+l+1}  \! \! \ d {\hat{y}}
    \end{align*}
Using a transformation $D_2: (-1,1) \rightarrow (0,1)$ for the second integral and \eqref{eq:jacobiortho} we see that the first integral vanishes for $i \neq l$ and the second integral for $i = l$, and so
    \begin{align*} 
                  (\phi^{F}_{i,j} \!\cdot \!n, \phi^{F}_{l,m} \!\cdot\! n)_F = c(i,j,l,m) \delta_{il} \delta_{jm} \quad \forall i,l \ge 1, i+j, l+m \le k.
    \end{align*}
    For $i = l =  0$ we have with $\mathcal{N}_0^{[f_1,f_2]} = \frac{1}{2} (y+z-1, -x, -x)$, thus on $F$ as $z = 0$,
    \begin{align*}
      \phi_{0j}^{F}\big{|}_F \cdot n = \nabla \times (\mathcal{N}_0^{[f_1,f_2]} v_{2j}^F ) \cdot n 
      =  ~\intjacobi_j^{3}\left(2y-1\right) - (1-y) \jacobi_{j-1}^{3}\left(2y-1\right).
    \end{align*}
    Using the Duffy transformation $D_1$ and $D_2$ we have as before
    \begin{align*}
                                 (\phi^{F}_{0j}& \! \cdot \! n, \phi^{F}_{0m}\! \cdot\! n)_F
      =   \int_{-1}^1 (1-\hat{y})  \intjacobi_j^{3}(\hat{y}) \intjacobi_m^{3}(\hat{y}) - \frac{(1-\hat{y})^2}{2}\left( \intjacobi_j^{3}(\hat{y})  \jacobi_{m-1}^{3}(\hat{y}) + \intjacobi_m^{3}(\hat{y})  \jacobi_{j-1}^{3}(\hat{y}) \right) \\
      &~~~~~~~~~~~~~~~~~~~~~~~~~~~~~~~~~ + \frac{(1-\hat{y})^3}{4} \jacobi_{j-1}^{3}(\hat{y}) \jacobi_{m-1}^{3}(\hat{y}) \ d {\hat{y}}.
    \end{align*}
    The integral of the last term is equal to $\delta_{jm}$ due to property \eqref{eq:jacobiortho}. For the first of the mixed terms we get with integration by parts and $\intjacobi_j^{3}(-1) = 0$
    \begin{align*}
      - \int_{-1}^1 \frac{(1-\hat{y})^2}{2} \intjacobi_j^{3}(\hat{y})  \jacobi_{m-1}^{3}(\hat{y}) \ d {\hat{y}} = \int_{-1}^1 \Big(-(1-\hat{y}) \intjacobi_j^{3}(\hat{y}) + \frac{(1-\hat{y})^2}{2} \jacobi_{j-1}^{3}(\hat{y}) \Big) \intjacobi_{m}^{3}(\hat{y}) \ d {\hat{y}},
    \end{align*}
 Putting all terms together we have
     \begin{align*}
           (\phi^{F}_{0j} \! \cdot \! n,\phi^{F}_{0m} \! \cdot \!n)_F = c(j,m) \delta_{jm} \quad 1 \le j,m \le k
     \end{align*}
     The orthogonality between functions with with $i = 0$ and $i \ge 1$  follows similar by using the Duffy transformation and \eqref{eq:jacobiortho}. The orthogonality with respect to the lowest order Raviart Thomas functions $\phi^{F}_{00}$ follows with the same arguments as in the two dimensional case. The other faces follow analogously.
   \end{proof} \\
   \begin{lemma} \label{lem:facetortho}
     The highest order facet basis functions from section~\ref{section:localspaces} are $L^2(\hat{T})$-orthogonal on polynomials up to degree $k-2$.
There holds
  \begin{align*}
    ( \phi_{\ast}^F , q)_{\widehat{T}} = 0 \quad  \forall q \in [\mathcal{P}^{k-2}(\widehat{T})]^d, \quad \forall F \in \partial \widehat{T}.
  \end{align*}
  for $\phi_{\ast}^F = \phi_k^F$ (in two dimensions) and $\phi_{\ast} \in \{\phi_{ij}^k\}_{i+j=k}$ (in three dimensions).
\end{lemma}
\begin{proof}
  We start with the two dimensional case and the lower edge $F = [1,2]$. The highest order edge basis functions is given by $\phi_k^{[1,2]} = \curl ( \intjacobi_{k+1}^{0} ( \frac{x}{1-y}) (1-y)^{k+1} )$, thus
  \begin{align*}
\phi_k^{[1,2]} = \left(\jacobi_{k}^0 ( \frac{x}{1-y}) x (1-y)^{k-1} - \intjacobi_{k+1}^0( \frac{x}{1-y}) (k+1) (1-y)^{k}, - \jacobi_{k}^0( \frac{x}{1-y}) (1-y)^{k} \right). 
  \end{align*}
Using a monomial basis for $\mathcal{P}^{k-2}(\widehat{T})$, so $x^my^n$ with $m+n \le k-2$, we get for the first component
\begin{align*}
  \int_{\widehat{T}} (\phi_k^F)_1 &x^my^n \ d {(x,y)} \\
  &= \int_{\widehat{T}} \left(\jacobi_{k}^0 ( \frac{x}{1-y}) x (1-y)^{k-1} - \intjacobi_{k+1}^0( \frac{x}{1-y}) (k+1) (1-y)^{k} \right) x^my^n \ d {(x,y)}.
\end{align*}
With the Duffy transformation $D_1$ as in the proof of Lemma \ref{lem:normalltwoorthogonal} we get 
\begin{align*}
  \int_{\widehat{T}} &(\phi_k^F)_1 x^my^n \ d {(x,y)} 
                     =\int_{-1}^1 \jacobi_{k}^0( \hat{x}) \hat{x}^{m+1} \ d {\hat{x}} \int_0^1 (1-\hat{y})^{k+m+1} \hat{y}^n \ d {\hat{y}} \\
  & ~~~~~~~~~~~~~~~~~~~~~~~~~~~~~~~~- (k+1) \int_{-1}^1\intjacobi_{k+1}^0( \hat{x}) \hat{x}^{m} \ d {\hat{x}}  \int_0^1 (1-\hat{y})^{k+m+1} \hat{y}^n   \ d {\hat{y}} = 0,
\end{align*}
due to $m \le k-2$ and equations \eqref{eq:jacobiortho} and \eqref{eq:intjacobiortho}. For the second component we proceed similar. For the proof in the three dimensional case let $F =[1,2,3]$ and start with $i \ge 1$ such that $i+j = k$. The first component of $-\phi_{ij}^{F} = \nabla u_{i+1}^F \times \nabla v_{(i+1)(j+1)}^F $ is given by
\begin{align*}
  (-\phi_{ij}^F)_1 =& \partial_y u^F_{(i+1)} \partial_z v_{(i+1)(j+1)}^F - \partial_z u^F_{(i+1)} \partial_y v_{(i+1)(j+1)}^F =a c - b c,
\end{align*}
with $\partial_y u^F_{(i+1)} = \partial_z u^F_{(i+1)} = a-b$ and $c = \partial_z v_{(i+1)(j+1)}^F - \partial_y v_{(i+1)(j+1)}^F$, thus
\begin{align*}
  a &:= \jacobi_i^0\Big(\frac{x}{1-y-z}\Big) x (1-y-z)^{i-1} \qquad   b := \intjacobi_{i+1}^{0}\Big(\frac{x}{1-y-z}\Big)(i+1)(1-y-z)^i
\end{align*}
and
\begin{align*}
    c& :=\Big(\jacobi_j^{2i+1}\Big(\frac{2y+z-1}{1-z}\Big) 2y (1-z)^{j-1} - \intjacobi_{j+1}^{2i+1}\Big(\frac{2y+z-1}{1-z}\Big) (j+1) (1-z)^{j} \\
& ~~~~~~~~~~~~~~~~~~~~~~~~~~~~~~~~~~~~~~~~~~~~~~~~~~~~~~~~~~~~~~~~ - \jacobi_j^{2i+1}\Big(\frac{2y+z-1}{1-z}\Big) 2 (1-z)^{j}\Big).
\end{align*}
Using a Duffy transformation $D_3: (-1,1) \times (0,1) \times (0,1) \rightarrow {\widehat{T}}$, with $D_3(\hat{x}, \hat{y}, \hat{z}) = (\hat{x} (1-\hat{y})(1-\hat{z}), \hat{y}(1-\hat{z}), \hat{z})$ we get for the integral of $ac$ multiplied with a monome $x^m y^nz^r$ with $m+n+r \le k-2$
\begin{align*}
  \int_{\widehat{T}} &ac~ x^m y^nz^r \ d {(x,y,z)}
= \int_{-1}^1 \jacobi_i^0(\hat{x})\hat{x}^{m+1}  \ d \hat{x} \int_{0}^1 (1-\hat{z})^{i+j+m+n+2}\hat{z}^r  \ d \hat{z}\\
  &~~~~~~~~~~~~~\cdot \int_0^1 - (1 - \hat{y})^{i+m+1} \hat{y}^n\left(2 (1 - \hat{y}) \jacobi_j^{2i+1}(2 \hat{y} -1) +  (j+1) \intjacobi_{j+1}^{2i+1}(2 \hat{y} -1) \right) \ d \hat{y}.
\end{align*}
For $m \le i-2$ the integral with respect to $\hat{x}$ vanishes, so it remains the case ${m>i-2} \Leftrightarrow i \le m + 1$. For the integral with respect to $\hat{y}$ we get using integration by parts
\begin{align*}
  &\int_0^1 - (1 - \hat{y})^{i+m+1} \hat{y}^n\left(2 (1 - \hat{y}) \jacobi_j^{2i+1}(2 \hat{y} -1) +  (j+1) \intjacobi_{j+1}^{2i+1}(2 \hat{y} -1)\right) \ d \hat{y} \\
    &=\int_0^1 \! - 2(1 - \hat{y})^{i+m+2} \hat{y}^n \jacobi_j^{2i+1}(2 \hat{y} -1) \!\ d \hat{y} \! - \! \! \int_0^1 (j+1) (1 - \hat{y})^{i+m+1} \hat{y}^n \intjacobi_{j+1}^{2i+1}(2 \hat{y} -1) \ d \hat{y} \\
  &=\int_0^1 - 2(1 - \hat{y})^{i+m+2} \hat{y}^n \jacobi_j^{2i+1}(2 \hat{y} -1)\ d \hat{y}  \\
&~~~~~~~~~~~~-  \frac{j+1}{i+l+2} \int_0^1 (1 - \hat{y})^{i+m+2}(2 \hat{y}^n \jacobi_j^{2i+1}(2 \hat{y} -1) + \intjacobi_{j+1}^{2i+1}(2 \hat{y} -1)n \hat{y}^{n-1}) \ d \hat{y}.
\end{align*}
Now as $2i +1 \le i + m + 2$ and $(i+m+2) - (2i + 1) + n \le j -1$ we can write $ \hat{y}^n(1 - \hat{y})^{i+m+2} = (1 - \hat{y})^{2i+1} w(\hat{y})$ where $w$ is polynomial of order less or equal $j-1$. By this the integrals of the Jacobian polynomials vanish due to \eqref{eq:jacobiortho}, and with a similar argument also the integral of the integrated Jacobian polynomial vanish due to \eqref{eq:intjacobiortho}. Using the same techniques one shows that also the integral $\int_{\widehat{T}} bc~ x^m y^nz^r \ d {(x,y,z)}$ vanishes what leads to the orthogonality of the first component $(\phi_{ij}^F)_1$. In a similar way one proves the orthogonality also for the other two components. It remains the proof for the face based basis functions with $i=0$, thus $\phi_{0k}^{F} := \nabla \times (\mathcal{N}_0^{[f_1,f_2]} v_{2k}^F )$. We start with the second component
\begin{align*}
  (\phi_{0k}^{F})_2 &
                    = \intjacobi^3_k(\frac{2y+z-1}{1-z}) (1-z)^k + (y+z-1) \jacobi^3_{k-1}(\frac{2y+z-1}{1-z}) y (1-z)^{k-2} \\
  & ~~~~~~~~~~~~~~~~~~~~~~~~- \frac{k}{2}(y+z-1) \intjacobi^3_k(\frac{2y+z-1}{1-z}) (1-z)^{k-1}.
\end{align*}
Again using the Duffy transformation $D_3$ and the monome $x^m y^nz^r$ with $m+n+r \le k-2$ we get
\begin{align*}
  \int_{\widehat{T}} & (\phi_{0k}^{F})_2~ x^m y^nz^r \ d {(x,y,z)}= \int_{-1}^1 \hat{x}^m \ d \hat{x} \int_0^1 (1-\hat{z})^{k+m+n+2}\hat{z}^r \ d \hat{z} \\
                                                                 &\cdot \Big( -\int_0^1 \hat{y}(1-\hat{y})^{m+2}\hat{y}^{n}  \jacobi^3_{k-1}(2 \hat{y} -1)   \ d \hat{y}+ \int_0^1 (1-\hat{y})^{m+1}\hat{y}^n \intjacobi^3_k(2 \hat{y} -1)  \ d \hat{y} \\
                                                                 &~~~~~~~~~~~~~~~~~~~~~~~~~~~~~~~~~~~~~~~~~~~~~~~~~~~~~~- \frac{k}{2} \int_0^1 (1-\hat{y})^{m+2}\hat{y}^n \intjacobi^3_k(2 \hat{y} -1)    \ d \hat{y} \Big).
\end{align*}
For $m=1$ the integral with respect to $\hat{x}$ vanishes, and for $m>1$ all integrals with respect to $\hat{y}$ vanish due to \eqref{eq:jacobiortho} and \eqref{eq:intjacobiortho}. For $m = 0$ we proceed similar as for the proof of the first type of face based basis functions. First apply integration by parts for the second integral with respect to $\hat{y}$
\begin{align*}
  \int_0^1 (1-\hat{y})\hat{y}^n &\intjacobi^3_k(2 \hat{y} -1)  \ d \hat{y} =  \int_0^1 \frac{(1-\hat{y})^{2}}{2} \Big(n \hat{y}^{n-1} \intjacobi^3_k(2 \hat{y} -1) + \hat{y}^n \jacobi^3_{k-1}(2 \hat{y} -1) \Big)  \ d \hat{y}.
\end{align*}
Adding up all integrals then gives
\begin{align*}
  \int_{\widehat{T}} (\phi_{0k}^{F})_2~ x^m y^nz^r & \ d {(x,y,z)}
  \\
  =&-\int_0^1 \hat{y}(1-\hat{y})^{2}\hat{y}^{n}  \jacobi^3_{k-1}(2 \hat{y} -1) \ d \hat{y} + \int_0^1 (1-\hat{y})^{2}\hat{y}^{n}  \jacobi^3_{k-1}(2 \hat{y} -1) \ d \hat{y} \\
 & - \frac{k}{2} \int_0^1 (1-\hat{y})^{2}\hat{y}^n \intjacobi^3_k(2 \hat{y} -1)    \ d \hat{y} + \frac{n}{2} \int_0^1 (1-\hat{y})^{2}\hat{y}^{n-1} \intjacobi^3_k(2 \hat{y} -1)    \ d \hat{y}.
\end{align*}
The first two integrals can be summed up leading to a coefficient $(1-\hat{y})^{3}$, thus vanish due to \eqref{eq:jacobiortho}. For the other two terms one uses integration by parts once again and uses \eqref{eq:intjacobiortho} to finally conclude the orthogonality of the second component. For the other two components $(\phi_{0k}^{F})_1$ and $(\phi_{0k}^{F})_3$ one proceeds similar, what finishes the proof.
\end{proof}


\subsection{Construction of the global FE Spaces} \label{section:globalspaces}
As usual with $hp$ finite elements, we associate element basis functions in $\HdivSpace$ (or $\HdivSpaceHODC$) with facets or cells. Cell type basis functions have zero normal trace while the normal trace of the basis functions associated with one facet span the polynomial space up to order $k$. The novel choice of the element basis functions is such that their normal component form a hierarchical $L^2$ orthogonal basis on facets. In three dimensions facet basis functions depend on indices $i,j$, i.e. $\phi^F_{ij}$ (for the two-dimensional case we remove the index $j$, i.e. $\phi^F_i$), and we have
\begin{align*}
  \hull(\{ \phi^F_{ij} \cdot n : i+j \le k \} ) = \mathcal{P}^k(F) \quad \text{and} \quad 
  \hull(\{ \phi^F_{ij} \cdot n : i+j = k \} ) \perp \mathcal{P}^{k-1}(F).
\end{align*}
This follows due to the hierarchical structure of integrated Jacobi polynomials and Lemma \ref{lem:normalltwoorthogonal} on the reference element. By applying the Piola transformation this property carries over to the physical elements in the mesh.
\begin{remark}
  The normal trace of the Piola mapped basis functions are biorthogonal to trivially mapped basis functions, also on curved elements.
\end{remark}\\
To obtain the different spaces $\HdivSpace$ and $\HdivSpaceHODC$ we use the same local basis functions but put them together differently to define the global functions. 
For the construction of $\HdivSpace$ we associate all the degrees of freedom to facet basis functions $\phi_{ij}^F$ with mesh facets to obtain normal continuity. 
For the space $\HdivSpaceHODC$ with relaxed $H(\divergence)$-conformity we only associate degrees of freedom with $i + j < k$ with facets. The remaining degrees of freedom with $i+j = k$ are associated with cells and are treated as locally (on only one cell) supported functions.
This means that the every facet basis functions in $\HdivSpace$ with $i+j = k$ is split up into two (locally supported) basis functions: For a facet $F\in\facets$ and $i+j=k$ the basis function $ \phi_{ij}^F $ with $\operatorname{supp}(\phi_{ij}^F) = T \cup T'\in\mesh$ is replaced with the basis functions $\phi_{ij}^{F,T} = \phi_{ij}^{F}|_T$ and $\phi_{ij}^{F,T'} = \phi_{ij}^{F}|_{T'}$ which have $\operatorname{supp}(\phi_{ij}^{F,T}) = T$ and $\operatorname{supp}(\phi_{ij}^{F,T'}) = T'$.
We note that the new functions $\phi_{ij}^{F,T}$ and $\phi_{ij}^{F,T'}$ (with non-zero normal trace on $F$) can be treated as interior bubble functions.
From this construction we observe that a realization of the non-standard space $\HdivSpaceHODC$ is not more difficult than the construction of $\HdivSpace$.

The split-up functions form bases for the local and global spaces
\begin{equation}\label{eq:wstarspace}
  \HdivSymb^\ast_T := \bigoplus_{F\in \partial T} \hull(\{ \phi^{F,T}_{ij} : i+j = k \} ), \quad T\in\mesh, \qquad   \HdivSymb^\ast_h := \bigoplus_{T\in\mesh} \HdivSymb^\ast_T.
\end{equation}
Note that Lemma \ref{lem:facetortho} implies that
\begin{equation}
  \int_{\Omega} \phi \cdot q d x = 0 \quad  \text{ for all } \phi \in \HdivSymb^\ast_h,~ \text{ for all } ~q \in \bigoplus_{T\in\mesh} [\mathcal{P}^{k-2}(T)]^d.
\end{equation}

\subsection{A simple and fast realization of the BDM interpolation} \label{hdgnse:stokes:reconstruction:average}
Due to the different association of the degrees of freedom of facet based basis functions (see section \ref{section:globalspaces}) a discontinuity of the normal component of a function $\HdivVar \in \HdivSpaceHODC$ only appears in the highest order components.
The idea of the reconstruction operator $\ReconHdiv$ is to exploit the origin of this discontinuity in the basis functions. It simply reverts the break-up of one basis function into two by applying an average on each facet of the corresponding coefficients. Let $F \in \facets$ be an arbitrary facet with $T, T' \in \mesh$ such that $F =  T \cap T'$. In three dimensions we have the local representation (in two dimensions we again remove the index $j$) of the normal components
\begin{align*}
  (\HdivVarT \cdot n)|_F &= \sum\limits_{i+j < k} \alpha_{ij} \phi_{ij}^F \cdot n + \sum\limits_{i+j =  k} \alpha^{T\hphantom{'}}_{ij}  \phi_{ij}^{F,T\hphantom{'}} \cdot n \\
  (\HdivVarTT \cdot n)|_F &= \sum\limits_{i+j < k} \alpha_{ij} \phi_{ij}^F \cdot n + \sum\limits_{i+j =  k} \alpha^{T'}_{ij}  \phi_{ij}^{F,T'} \cdot n,
\end{align*}
and its jumps
$\jumpDG{ \HdivVar \cdot n } = \sum_{i+j =  k} (\alpha^{T}_{ij} - \alpha^{T'}_{ij})  \phi_{ij}^{F}\cdot n$.
Here $\alpha_{ij}, \alpha_{ij}^{T}$ and $\alpha_{ij}^{T'}$ are given coefficients, and $\phi_{ij}^{F,T}, \phi_{ij}^{F,T'} $ are the highest order basis functions associated to the elements $T$ and $T'$. Note that $ \phi_{ij}^{F}\cdot n = \phi_{ij}^{F,T}\cdot n = \phi_{ij}^{F,T'}\cdot n$ on $F$ for $i+j=k$. With $\overline{\alpha}_{ij} = \frac{1}{2} (\alpha^{T}_{ij} + \alpha^{T'}_{ij})$ we define the reconstruction operation $\ReconHdiv \HdivVar$ such that for $\tilde{T} \in \{ T, T' \}$ we have
\begin{align} \label{eq:aver}
(\ReconHdiv \HdivVar|_{\tilde{T}} \cdot n)|_F := \!\!\sum\limits_{i+j < k} \alpha_{ij} \phi_{ij}^F \cdot n+ \!\!\sum\limits_{i+j =  k} \overline{\alpha}_{ij} \phi_{ij}^{F}\!\!\cdot n.
\end{align}
In the reconstruction we thus only apply this averaging which only affects the highest order facet degrees of freedom. Hence, the reconstruction can also be characterized as a perturbation in the space $\HdivSymb^\ast_h$.
In Figure \ref{fig:aver} we see a sketch of the averaging on one facet of two neighboring elements.
\begin{figure}[ht]
  {
    \begin{center}
      \begin{tikzpicture}
        \node at (0,0) {\includegraphics[width=0.4\textwidth]{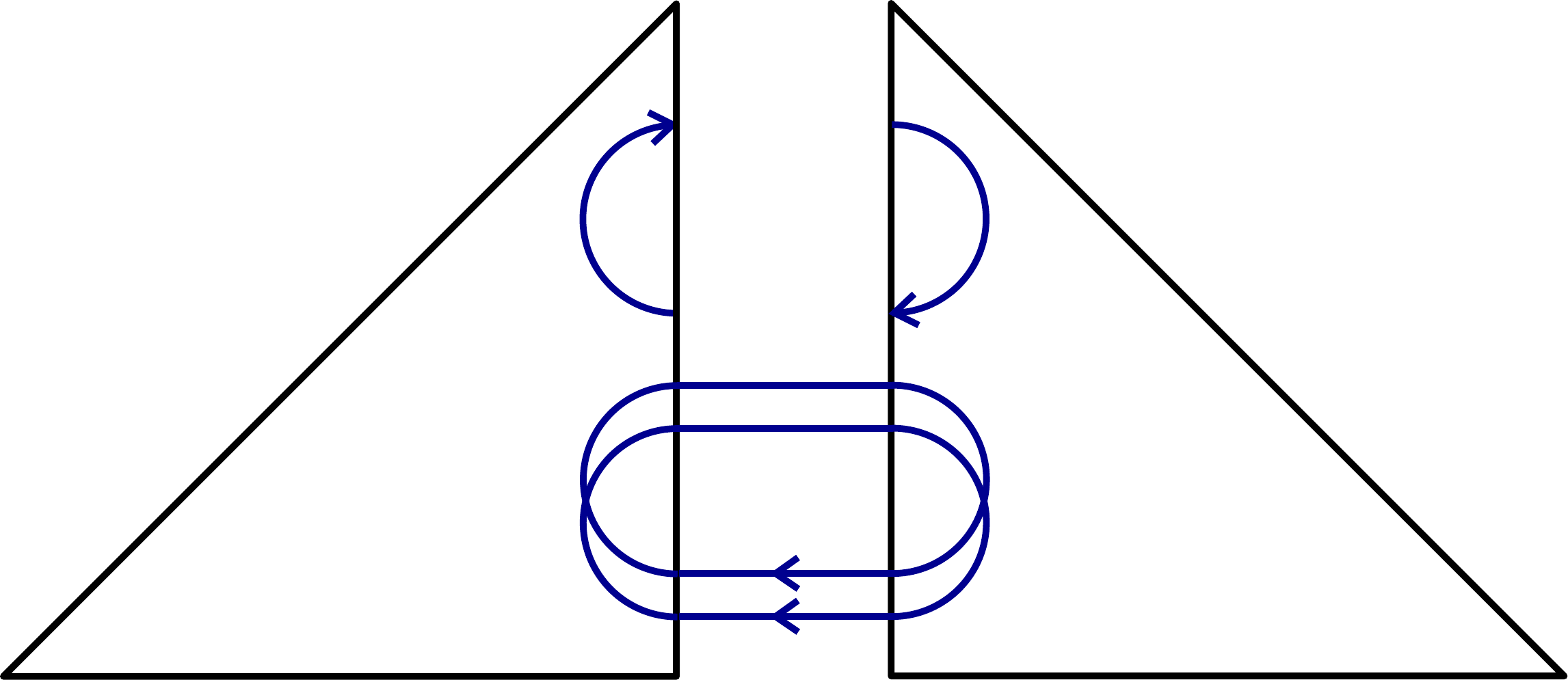}};
        \node at (-1.5,-0.7) {$T$};
        \node at (1.5,-0.7) {$T'$};         
        \draw[->] (2.5,0.2)[bend left=00] to node[above] {$\ReconHdiv$}  (5,0.2);
        \node at (7.6,0) {      \includegraphics[width=0.4\textwidth]{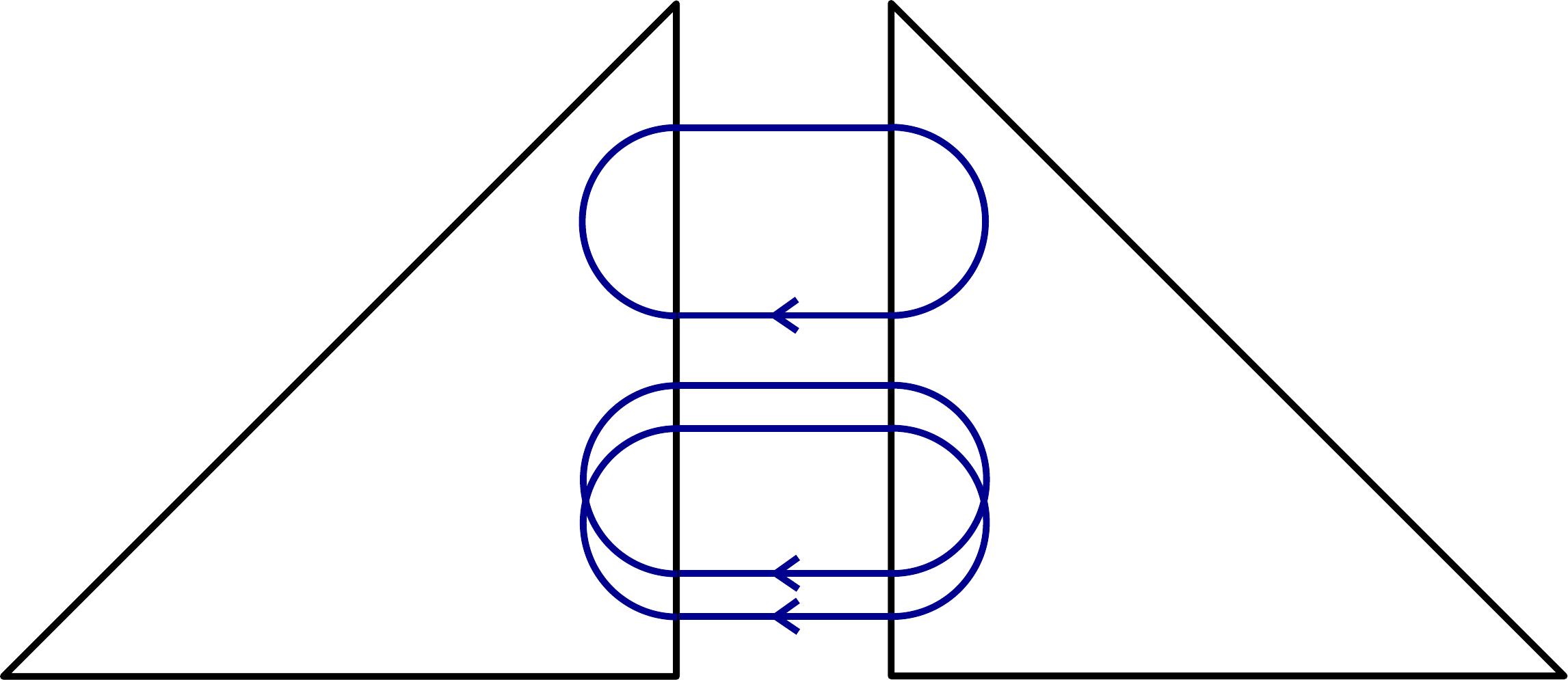}};
        \node at (6.5,-0.7) {$T$};
        \node at (9.5,-0.7) {$T'$};
        \node at (-2,1.5) {$\alpha^{T}_{ij} \phi_{ij}^{F,T}$};
        \draw[->] (-2,1.1)[bend right=30] to (-0.7,0.4);
        \draw[->] (2,1.1)[bend left=30] to (0.7,0.4);
        \node at (2,1.5) {$\alpha^{T'}_{ij} \phi_{ij}^{F,T'}$};
        \node at (6,1.5) {$\overline{\alpha}_{ij}  ( \phi_{ij}^{F,T} + \phi_{ij}^{F,T'}) = \overline{\alpha}_{ij} \phi_{ij}^{F}$};
        \draw[->] (6,1.1)[bend right=30] to (7.3,0.4);
        
      \end{tikzpicture}
    \end{center}
    \caption{In the left picture we see two highest order facet functions $\phi_{ij}^{F,T}$ and $\phi_{ij}^{F,T'}$ and the continuous low order facet functions of two neighboring triangles $T$ and $T'$. The reconstruction $\ReconHdiv$ uses the average value of the corresponding coefficients to remove the normal discontinuity in the highest order, see on the right.}
  } 
  \label{fig:aver}
\end{figure}
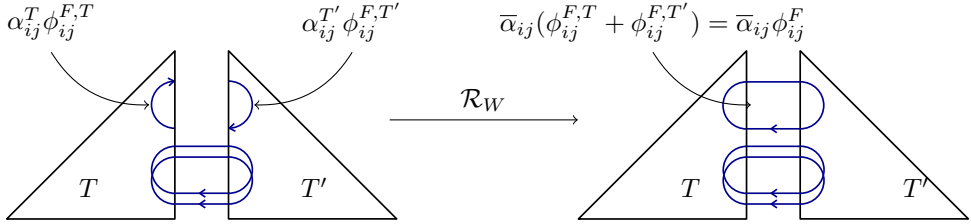
\begin{lemma}
The reconstruction operator $\ReconHdiv$ defined by the averaging \eqref{eq:aver} fulfills Assumption \ref{ass:recon}. 
\end{lemma}
\begin{proof} \eqref{eq:ass1} follows by construction, \eqref{eq:ass2a} follows directly from Lemma \ref{lem:normalltwoorthogonal} and \eqref{eq:ass2b} from Lemma \ref{lem:facetortho}. It remains to show \eqref{eq:ass3}.
Let $\facets(T)$ denote the set of facets surrounding $T \in \mesh$ and $\localnorm{\cdot}$ be the element-local version of $\brokenHnorm{\cdot}$ with
  $
  \localnorm{\VelTest}^2 := \Vert \nabla \HdivTestT \Vert_T^2 + \frac{1}{h} \Vert \facetproj^{k-1} \facetjump{\VelTest^t} \Vert_{\partial T}^2.
  $
We have
  \begin{equation}
    \localnorm{ \Recon \VelTest } \leq \localnorm{w_h} + \localnorm{ \VelTest },
  \end{equation}
  with $w_h = (w_{\mathcal{T}},w_{\mathcal{F}}) := \Recon \VelTest - \VelTest$ and notice that $w_T \in \HdivSymb_T^\ast$, $w_\mathcal{F} = 0$, cf. \eqref{eq:wstarspace}.
  With transformation to the reference element and equivalence of finite dimensional norms, one easily shows that the norms $\localnorm{\cdot}$, $\Vert \nabla \cdot \Vert_T$ and $\localnormII{\cdot}$ with 
  $
  \localnormII{ \VelTest }^2 := \frac{1}{h} \Vert (\id - \facetproj^{k-1}) (\HdivTestT \cdot n) \Vert_{\partial T}^2
  $
  are equivalent on $\HdivSymb_T^\ast \times \{0\}$ with constants independent of $h$. Hence, 
  $$
  \localnorm{ \Recon \VelTest }^2 \lesssim  \localnormII{ w_h }^2 + \localnorm{\VelTest}^2 \quad  \Longrightarrow \quad 
  \brokenHnorm{\Recon \VelTest}^2  \lesssim  \brokenHnorm{\VelTest}^2 + \sumoverallelements \localnormII{ w_h }^2.
  $$
  Due to the averaging in \eqref{eq:aver} we have
  \begin{align*}
    \sumoverallelements \localnormII{ w_h }^2
    & = \sumoverallelements \frac{1}{h} \Vert (\id - \facetproj^{k-1}) (w_T \cdot n) \Vert_{\partial T}^2
      \lesssim 
      \sumoverallelements \frac{1}{h} \Vert (\id - \facetproj^{k-1}) (\HdivTestT \cdot n) \Vert_{\partial T}^2. 
  \end{align*}
  With a simple Bramble-Hilbert argument we obtain $\frac{1}{h} \Vert (\id - \facetproj^{k-1}) (\HdivTestT \cdot n) \Vert_{\partial T}^2 \lesssim \Vert \nabla \HdivTestT \Vert_{T}^2$
  from which we conclude \eqref{eq:ass3}.
\end{proof}

\section{A priori error analysis} \label{sec:errorana}
In this section we consider the a priori analysis of the discretizations \eqref{eq:discstokes}, \eqref{eq:probuststokes}, both with and without a subsequent application of the reconstruction operator $\Recon$ that fulfills Assumption \ref{ass:recon}.
\subsection*{Preliminaries}
In order to compare discrete velocity functions $\VelVar = (\HdivVar,\FacetVar) \in U_h$ with functions $u \in \Ureg := [H_0^1(\Omega) \cap H^2(\mesh)]^d$ we identify (with abuse of notation) $u$ with the tuple $(u,u|_{\mathcal{F}})$ where $u|_{\mathcal{F}}$ is to be understood in the usual trace sense (which is unique due to the $H^1(\Omega)$ regularity). For the purpose of the analysis it is convenient to introduce the big bilinearform for the saddle point problem in \eqref{eq:discstokes} for $(\VelVarEx,\PressureVarEx), (\VelTestEx,\PressureTestEx) \in (\VelSpace  + \Ureg) \times L^2(\Omega)$:
\begin{align}
  \bilinearform{K}((\VelVarEx,\PressureVarEx),(\VelTestEx,\PressureTestEx)) & := \bilinearform{A}(\VelVarEx,\VelTestEx) +  \bilinearform{B}(\VelVarEx,\PressureTestEx) + \bilinearform{B}(\VelTestEx,\PressureVarEx). 
\end{align}
On top of the discrete norms for $\VelVar \in \VelSpace$ and $\PressureVar \in \PressureSpace$, cf. \eqref{eq:discretenorms}, we introduce the following stronger norms for $\VelVarEx \in \VelSpace + \Ureg$ and $\PressureVarEx \in L^2(\Omega)$:
\begin{align*}
  \brokenHnormII{\VelVarEx}^2
  \!:= \! \brokenHnorm{\VelVarEx}^2 + \sumoverallelements h \Vert \nabla \HdivVarEx \Vert_{\partial T}^2, \qquad \brokenHnormII{(\VelVarEx,\PressureVarEx)} 
  \!:= \! \sqrt{\nu} \brokenHnormII{\VelVarEx} + \frac{1}{\sqrt{\nu}}\Vert \PressureVarEx \Vert_{L^2(\Omega)},\ 
\end{align*}
These stronger norms $\brokenHnormII{\cdot}$ allow to control the normal derivatives also for the exact solution.

\subsection{Analysis of the basic method \eqref{eq:discstokes}}
In this subsection we take a closer look at the analysis of the \emph{basic} discretization method \eqref{eq:discstokes}, cf. section \ref{sec:hdgstokes}. The analysis follows standard Strang-type arguments based on consistency,  continuity, inf-sup stability (induced by coervitiy of $\bilinearform{A}(\cdot,\cdot)$ and LBB-stability of $\bilinearform{B}(\cdot,\cdot)$).

%
\begin{lemma}[Consistency] \label{lem:consistency}
  Let $(\VelVarEx, \PressureVarEx) \in \Ureg \times L^2_0(\Omega)$ be the solution to the Stokes equation \eqref{eq:stokes}.
  There holds for $(\VelTest,\PressureTest) \in \VelSpace \times \PressureSpace$
  \begin{align}
    \bilinearform{K}((\VelVarEx, \PressureVarEx),(\VelTest,\PressureTest))
    & = \int_{\Omega} f \, \HdivTest \, dx -  \ConDiff(\VelVarEx, \PressureVarEx, \VelTest), \\
\text{ with } \quad       \ConDiff(\VelVarEx, \PressureVarEx, \VelTest)
    & := \sum_{T\in\mesh} \int_{\partial T} (\id - \facetproj^{k-1}) (- \nu \frac{\partial \VelVarEx}{\partial n} + \PressureVarEx n) \cdot (\id - \facetproj^{k-1}) \HdivTestT \, ds.
    \end{align}
    For $(\VelVarEx, \PressureVarEx) \in  [H^1(\Omega)]^d \cap [H^l(\mesh)]^d \times H^{l-1}(\mesh),~l \ge 2$ and $m = \min(k,l-1)$ we get
    \begin{align} \label{eq:consistencyerror}
    \ConDiff(\VelVarEx, \PressureVarEx, \VelTest)
     \le h^m \left( \sqrt{\nu} \|  \VelVarEx \|_{H^{m+1}(\mesh)} + \frac{1}{\sqrt{\nu}} \|  \PressureVarEx \|_{H^{m}(\mesh)} \right) \,\sqrt{\nu}\, \brokenHnormleft v_h \brokenHnormright.
  \end{align}
\end{lemma}
\vspace*{-0.25cm}
\begin{proof}
  As $\divergence(\VelVarEx)=0$ and $\facetjump{ \VelVarEx^t }=0$ we get by an element wise partial integration
  \begin{align*}
    \bilinearform{K}& ((\VelVarEx,\PressureVarEx),(\HdivTest,\PressureTest))
    \\
&  =   \sumoverallelements
   \int_{T} ( \divergence(- \nu \nabla \VelVarEx) + \nabla \PressureVarEx) \HdivTestT \ d {x}
 + \int_{\partial T} \nu \frac{\partial \VelVarEx}{\partial {\normal} }  (\HdivTestT - \facetproj^{k-1} \facetjump{\VelTest^t}) 
         -\PressureVarEx \HdivTestT \cdot n \ d {s}
    \\
&  =
    \sumoverallelements 
   \int_{T} f \HdivTestT \ d {x}
   + \int_{\partial T} \nu \frac{\partial \VelVarEx}{\partial {\normal} }  (\HdivTestT^t - \facetproj^{k-1} \HdivTestT^t + \FacetTest^t) 
-(-\nu \frac{\partial \VelVarEx}{\partial {\normal} } \cdot n + \PressureVarEx ) (\HdivTestT^n \cdot n) \ d {s}.
  \end{align*}
  On interior facets we have $\jumpDG{\nu \frac{\partial \VelVarEx}{\partial {\normal} }} = 0$ and on boundary facets we have $\FacetTest^t = 0$ so that
  \begin{equation}
    \sumoverallelements \int_{\partial T} \nu \frac{\partial \VelVarEx}{\partial {\normal} } \FacetTest^t \ d {s}  = \sumoverallinnerfacets \int_F \jumpDG{\nu \frac{\partial \VelVarEx}{\partial {\normal} }} \FacetTestF^t \, ds +
    \sumoverallouterfacets \int_F \nu \frac{\partial \VelVarEx}{\partial {\normal} } \FacetTestF^t \, ds = 0.
  \end{equation}
  Further we have, with $\sigma_n = -\nu \frac{\partial \VelVarEx}{\partial {\normal} } \cdot n  + \PressureVarEx$
  \begin{align}
    \sumoverallelements \int_{\partial T} \sigma_n \HdivTestT^n \cdot n \ d {s}
    & =       \sumoverallelements \int_{\partial T} \sigma_n (\id - \facetproj^{k-1}) (\HdivTestT^n\cdot n) \ d {s}.
  \end{align}
  In the last step we used that on the interior facets we have $\facetproj^{k-1} \jumpDG{\HdivTest \cdot n} = 0$ due to $\HdivTest \in \HdivSpaceHODC$ and on the boundary facets we have $\facetproj^{k-1} \HdivTest \cdot n = 0$. For the proof of \eqref{eq:consistencyerror} we start with the estimate of the velocity part. Using the Cauchy Schwarz inequality and properties of the $L^2$ projection on one element $T \in \mesh$ we get for $m = \min(k,l-1)$
  \begin{align*}
    \int_{\partial T}
    & (\id - \facetproj^{k-1}) (- \nu \frac{\partial \VelVarEx}{\partial n} + \PressureVarEx n) \cdot (\id - \facetproj^{k-1}) \HdivTestT \, ds \\
    &\le \left( \nu \Vert (\id - \facetproj^{k-1}) \nabla \VelVarEx \Vert_{\partial T} + \Vert (\id - \facetproj^{k-1}) \PressureVarEx \Vert_{\partial T} \right) \ \Vert (\id - \facetproj^{k-1}) \HdivTestT \Vert_{\partial T} \\
    &\lesssim h^{m-1/2} \left( \sqrt{\nu} \Vert \nabla \VelVarEx \Vert_{H^m(T)} + \frac{1}{\sqrt{\nu}} \Vert \PressureVarEx \Vert_{H^m(T)} \right) \ h^{\frac12} \sqrt{\nu} \Vert \nabla \HdivTestT \Vert_T.
  \end{align*}
  Summing over all elements concludes the proof. The pressure estimate follows with similar techniques.
\end{proof}

\begin{lemma}[Continuity] \label{lem:continuity}
  There holds
  \begin{subequations}
    \begin{align}
      \bilinearform{A}(\VelVarEx,\VelTest)
      & \lesssim  \sqrt{\nu} \brokenHnormII{\VelVarEx} \ \sqrt{\nu} \brokenHnorm{\VelTest}
      && \forall \VelVarEx \in \VelSpace + \Ureg, \VelTest \in  \VelSpace, \label{eq:contA} \\
      \quad \text{and} \quad \bilinearform{B}(\VelVarEx,\PressureVarEx)
      & \lesssim  \sqrt{\nu} \brokenHnorm{u} \ \frac{1}{\sqrt{\nu}} \| q \|_0
      && \forall \VelVarEx \in \VelSpace + \Ureg, \PressureVarEx \in L^2(\Omega), \label{eq:contB}
    \end{align}
  \end{subequations}
  which implies for all 
  $
  (\VelVarEx,\PressureVarEx) \in \VelSpace + \Ureg \times L^2(\Omega),~
  (\VelTest,\PressureTest) \in \VelSpace \times \PressureSpace:
  $
  \begin{equation}
    \bilinearform{K}((\VelVarEx,\PressureVarEx),(\VelTest,\PressureTest))
    \lesssim  \brokenHnormII{(\VelVarEx,\PressureVarEx)} \brokenHnorm{(\VelTest,\PressureTest)}.
  \end{equation}
  \vspace*{-0.3cm}
\end{lemma}
\vspace*{-0.25cm}
\begin{proof}
  Using the Cauchy Schwarz inequality on each triangle we get
  \begin{align*}
    \bilinearform{A}(\VelVarEx,\VelTest) \leq
    & \sumoverallelements \Big\{ \nu \Vert {\nabla} {\HdivVarEx} \Vert_T  \Vert {\nabla} {\HdivTestT} \Vert_T + \nu  \Vert {\nabla} {\HdivVarEx} \Vert_{\partial T} \Vert \facetproj^{k-1} \jumpleft \VelTest^t  \jumpright \Vert_{\partial T}  \\
    &+ \nu \Vert {\nabla} {\HdivTestT} \Vert_{\partial T} \Vert \facetproj^{k-1} \jumpleft \VelVarEx^t  \jumpright \Vert_{\partial T}
      +  \nu \frac{\lambda}{h} \Vert \facetproj^{k-1} \facetjump{ \VelVarEx^t } \Vert_{\partial T}   \Vert \facetproj^{k-1} \jumpleft \VelTest^t  \jumpright \Vert_{\partial T} \Big\} .
  \end{align*}
  All terms except the third term on the right hand side can naturally be bounded by the element contributions of the norms  $\brokenHnormII{\VelVarEx}$ and $\brokenHnorm{\VelTest}$. 
  With an inverse inequality for polynomials and Young's inequality we also get a suitable bound for the third term:
  \begin{align*}
    \Vert {\nabla} {\HdivTestT} \Vert_{\partial T} \Vert \facetproj^{k-1} \facetjump{ \VelVarEx^t} \Vert_{\partial T}
    &\le \Vert {\nabla} {\HdivTestT} \Vert_T h^{-\frac12} \Vert \facetproj^{k-1} \jumpleft \VelVarEx^t \jumpright \Vert_{\partial T}
  \end{align*}
  Finally, with the Cauchy Schwarz inequality in $\mathbb{R}^{|\mesh|}$ \eqref{eq:contA} is proven. Property \eqref{eq:contB} also follows by simply using the Cauchy Schwarz inequality.
\end{proof}

\begin{lemma}[Coercivity] \label{lem:coercivity}  
  There exists a positive number $c_\stab \in \rr$ such that for the stabilization parameter $\stab > c_\stab$ there holds
  \begin{equation*}
    \bilinearform{A}(\VelVar,\VelVar) \gtrsim  \nu \brokenHnormleft \VelVar \brokenHnormright^2 \quad \forall \VelVar \in \VelSpace.
  \end{equation*}
\end{lemma}
\vspace*{-0.25cm}
\begin{proof} Using the Cauchy Schwarz inequality we first get 
  \begin{align*}
    \bilinearform{A}(\VelVar,\VelVar)
    & \ge \sumoverallelements \nu \Vert {\nabla} {\HdivVarT} \Vert_T^2 - \nu 2 \Vert   \frac{\partial {\HdivVarT}}{\partial {\normal} } \Vert_{\partial T} \Vert \facetproj^{k-1} \facetjump{ \VelVar^t } \Vert_{\partial T}    + \nu \frac{\stab}{h} \Vert \facetproj^{k-1} \facetjump{ \VelVar^t } \Vert_{\partial T}^2.
  \end{align*}
  For the second term we first use an inverse trace inequality (with constant $c_{{}_{T}}$, \cite{WARBURTON20032765}) for polynomials and Young's inequality, thus 
  \begin{align*}
    2 \Vert   \frac{\partial {\HdivVarT}}{\partial {\normal} } \Vert_{\partial T} \Vert \facetproj^{k-1} \facetjump{ \VelVar^t } \Vert_{\partial T} &
   \leq  \frac{1}{2} \Vert  \nabla \HdivVarT  \Vert_{T}^2 +  \frac{2 c_{{}_{T}}^2 }{h} \Vert \facetproj^{k-1} \facetjump{ \VelVar^t } \Vert_{\partial T}^2.
  \end{align*}
This leads to $\bilinearform{A}(\VelVar,\VelVar) \ge \nu c_\stab \brokenHnormleft \VelVar \brokenHnormright^2$
  with $c_\stab = 2 c_{{}_{T}}^2$.
\end{proof}

\begin{lemma}[LBB] \label{lem:lbb}
  There holds
  \begin{equation}\label{eq:lbb}
    \sup_{\VelVar \in \VelSpace} \frac{\bilinearform{B}(\VelVar,\PressureVar)}{\sqrt{\nu}\brokenHnormleft \VelVar \brokenHnormright} \gtrsim \frac{1}{\sqrt{\nu}} \Vert \PressureVar \Vert_{L^2}, \quad \forall \ \PressureVar \in \PressureSpace.
  \end{equation}
\end{lemma}
\vspace*{-0.25cm}
\begin{proof}
  We refer to \cite[Proposition 2.3.5]{lehrenfeld2010hybrid} where the DG analysis of \cite{hansbo2002discontinuous} is applied to the $H(\divergence)$-conforming Hybrid DG method. A polynomial robust LBB estimate for the two dimensional DG case can be found in \cite{LedererSchoeberl2016}.
\end{proof}

\begin{lemma}[Strang] \label{lem:strang}
  Let $(\VelVarEx,\PressureVarEx) \in \Ureg \times L_0^2(\Omega)$ be the exact solution of~\eqref{eq:stokes} and $(\VelVar,\PressureVar) \in \VelSpace \times \PressureSpace$ be the solution of \eqref{eq:discstokes}. Then there holds
  \begin{equation}
    \brokenHnormII{(\VelVarEx-\VelVar,\PressureVarEx-\PressureVar)} \lesssim \inf_{(\VelTest,\PressureTest) \in \VelSpace \times \PressureSpace} \brokenHnormII{(\VelVarEx-\VelTest,\PressureVarEx-\PressureTest)} +
\sup_{\VelTest \in \VelSpace} \frac{\ConDiff(\VelVarEx, \PressureVarEx, \VelTest)}{\brokenHnormII{\VelTest}}
  \end{equation}
\end{lemma}
\vspace*{-0.25cm}
\begin{proof}
  We use a standard Strang-type approach for the proof. We first apply a simple triangle inequality
  for $(\VelTest,\PressureTest) \in \VelSpace \times \PressureSpace$,
  $$
  \brokenHnormII{(\VelVarEx-\VelVar,\PressureVarEx-\PressureVar)}
  \leq
  \brokenHnormII{(\VelVarEx-\VelTest,\PressureVarEx-\PressureTest)}
  +
  \brokenHnormII{(\VelTest-\VelVar,\PressureTest-\PressureVar)}.
  $$
  With Lemmas \ref{lem:continuity}, \ref{lem:coercivity}, \ref{lem:lbb} we can apply Brezzi's theorem to obtain inf-sup stability (with a constant independent of $h$ and $\nu$) for the bilinearform $\bilinearform{K}$ on the discrete space $\VelSpace \times \PressureSpace$, so that
  \begin{align*}
    \brokenHnormII{ (\VelVar - \VelTest,
    & \PressureVar - \PressureTest)}
    \brokenHnormII{(\VelTestB,\PressureTestB)}
      \lesssim
      \bilinearform{K}((\VelVar-\VelTest,\PressureVar-\PressureTest),(\VelTestB,\PressureTestB)) \\
    & =
      \bilinearform{K}((\VelVarEx-\VelTest,\PressureVarEx-\PressureTest),(\VelTestB,\PressureTestB))
      +
      {\bilinearform{K}((\VelVar-\VelVarEx,\PressureVar-\PressureVarEx),(\VelTestB,\PressureTestB))}\\
    & \lesssim
    \brokenHnormII{ (\VelVarEx - \VelTest, \PressureVarEx-\PressureTest)}
    \brokenHnormII{ (\VelTestB,\PressureTestB)} + \ConDiff(\VelVarEx,\PressureVarEx,\VelTestB).
  \end{align*}
Dividing by $\brokenHnormII{(\VelTestB,\PressureTestB)}$, the claim then follows as $(\VelTest,\PressureTest)$ was arbitrary.
\end{proof}


\begin{lemma} \label{lem:interpol}
  For $\VelVarEx \in [H^1(\Omega)]^d \cap [H^l(\mesh)]^d,~l\geq 2$ there holds for $m = \min\{ l-1, k\}$
  \vspace*{-0.25cm}
  \begin{subequations}
  \begin{equation}
    \inf_{\VelTest\in\VelSpace} \brokenHnormII{\VelTest - \VelVarEx} \lesssim h^{m} \Vert \VelVarEx \Vert_{H^{m+1}(\mesh)}.
  \end{equation}
  If further $\PressureVarEx \in H^{l-1}(\mesh) \cap L^2_0(\Omega)$, then 
  \begin{equation}
    \inf_{(\VelTest,\PressureTest)\in\VelSpace\times\PressureSpace} \brokenHnormII{(\VelTest - \VelVarEx,\PressureTest - \PressureVarEx)} \lesssim h^{m} ( \sqrt{\nu} \Vert \VelVarEx \Vert_{H^{m+1}(\mesh)} + \frac{1}{\sqrt{\nu}}\Vert \PressureVarEx \Vert_{H^{m}(\mesh)})
  \end{equation}
\end{subequations}
\end{lemma}
\vspace*{-0.25cm}
\begin{proof}
  The proof is build around usual Bramble-Hilbert type arguments, cf. \cite[Proposition 2.3.10]{lehrenfeld2010hybrid}.
\end{proof}

\begin{proposition} \label{prop:errest}
  Let $(\VelVarEx,\PressureVarEx) \in  [H^1(\Omega)]^d \cap [H^l(\mesh)]^d \times H^{l-1}(\mesh)\cap L^2_0(\Omega)$ for $l\geq 2$ be the exact solution of~\eqref{eq:stokes} and $(\VelVar,\PressureVar) \in \VelSpace \times \PressureSpace$ be the solution of \eqref{eq:discstokes}. Then for $m = \min\{l-1, k\}$ there holds 
  \begin{equation}
\brokenHnormII{(\VelVarEx-\VelVar,\PressureVarEx-\PressureVar)} \lesssim
 h^{m} ( \sqrt{\nu} \Vert \VelVarEx \Vert_{H^{m+1}(\mesh)} + \frac{1}{\sqrt{\nu}} \Vert \PressureVarEx \Vert_{H^{m}(\mesh)}).
  \end{equation}
\end{proposition}
\vspace*{-0.25cm}
\begin{proof}
Combine Lemma \ref{lem:consistency}, Lemma \ref{lem:strang} and Lemma \ref{lem:interpol}.
\end{proof}

\subsection{Analysis of the reconstruction operator} \label{analysis:recon}
With the solution of the \emph{basic} discretization \eqref{eq:discstokes} we obtain only weakly divergence free solutions, but can apply the reconstruction operator $\Recon$ in a subsequent step.
In the following we show that we thereby obtain solenoidal solutions while preserving optimal order convergence. 
\begin{lemma}\label{lem:recon}
  Let $\HdivVar \in \HdivSpaceHODC$ with $\divergence(\HdivVarT)=0$ for all $T\in\mesh$ and Assumption \ref{ass:recon} be true for $\ReconHdiv$. Then there holds
  \begin{equation}
    \ReconHdiv \HdivVar \in \HdivSpace \quad \text{ and } \quad \divergence(\ReconHdiv \HdivVar)=0.
  \end{equation}
\end{lemma}
\vspace*{-0.25cm}
\begin{proof}
  From \eqref{eq:ass1} we have that $\ReconHdiv \HdivVar  \in \HdivSpace$. To show that $\ReconHdiv \HdivVar$ is globally divergence free it only remains to show  $\divergence( (\ReconHdiv \HdivVar) |_T)=0$ for any element $T \in \mesh$.
  From partial integration and \eqref{eq:ass1}, \eqref{eq:ass2a} and \eqref{eq:ass2b} we have for $q \in \mathcal{P}^{k-1}(T)$ (with $\nabla q \in [\mathcal{P}^{k-2}(T)]^d$)
  \begin{subequations}
    \begin{align*}
      \int_{T} \divergence(\ReconHdiv \HdivVar ) q \, dx 
      & = - \int_{T} \ReconHdiv \HdivVar  \cdot \nabla q \, dx +  \int_{\partial T} \ReconHdiv \HdivVar \cdot n \, q \, ds \\
      & = - \int_{T} \HdivVarT \cdot \nabla q \, dx + \int_{\partial T} \HdivVarT \cdot n \, q \, ds
      = \int_{T} \divergence(\HdivVarT) q \, dx 
    \end{align*}
  \end{subequations}
From $\divergence(\HdivVarT)=0,~T\in\mesh$ we can conclude that the first term vanishes. Choosing $q =  \divergence(\ReconHdiv \HdivVar )$ concludes the proof.
\end{proof}

\begin{lemma} \label{lem:reconapprox}
  If $\Recon$ fulfills Assumption \ref{ass:recon}, there holds for $u \in \Ureg$:
  \begin{align}
    \brokenHnorm{\VelVarEx - \Recon \VelVar} & \lesssim \inf_{v_h \in \HdivSpace \times \FacetSpace} \brokenHnorm{\VelVarEx - \VelTest} + \brokenHnorm{\VelVarEx - \VelVar}.
  \end{align}
\end{lemma}
\vspace*{-0.25cm}
\begin{proof}
  We have that $\brokenHnorm{\Recon \VelTest } \lesssim \brokenHnorm{\VelTest}$ for every $\VelTest \in \VelSpace$. Hence, there holds
  \begin{align*}
    \brokenHnorm{\VelVarEx - \Recon \VelVar}
    & \leq \brokenHnorm{\VelVarEx - \VelTest} + \brokenHnorm{\Recon (\VelVar - \VelTest)} \\
    & \lesssim \brokenHnorm{\VelVarEx - \VelTest} + \brokenHnorm{\VelTest - \VelVar} 
      \lesssim \brokenHnorm{\VelVarEx - \VelTest} + \brokenHnorm{\VelVarEx - \VelVar},
  \end{align*}
for any $\VelTest \in \HdivSpace \times \FacetSpace \subset \VelSpace$.  
\end{proof}\\
\begin{proposition} \label{prop:errest2}
  Let $(\VelVarEx,\PressureVarEx) \in  [H_0^1(\Omega)]^d \cap [H^l(\mesh)]^d \times H^{l-1}(\mesh) \cap L^2_0(\Omega)$ for $l\geq 2$ be the solution of~\eqref{eq:stokes}, $(\VelVar,\PressureVar) \in \VelSpace \times \PressureSpace$ be the solution of \eqref{eq:discstokes} and $\Recon$ fulfill Assumption \ref{ass:recon}. Then there holds $\ReconHdiv \HdivVar \in H(\divergence,\Omega)$, $\divergence(\ReconHdiv \HdivVar) = 0$ and for $m = \min\{l-1, k\}$ 
  \begin{equation}
\brokenHnormII{(\VelVarEx-\Recon\VelVar,\PressureVarEx-\PressureVar)} \lesssim
 h^{m} ( \sqrt{\nu} \Vert \VelVarEx \Vert_{H^{m+1}(\mesh)} + \frac{1}{\sqrt{\nu}} \Vert \PressureVarEx \Vert_{H^{m}(\mesh)}).
  \end{equation}
\end{proposition}
\vspace*{-0.25cm}
\begin{proof}
  With Lemma \ref{lem:interpol} and Proposition \ref{prop:errest} this follows from Lemma \ref{lem:reconapprox}.
\end{proof}

\subsection{Analysis of the pressure robust method \eqref{eq:probuststokes}} \label{analysis:probust}
In this section we analyze the pressure robust method given by equation \eqref{eq:probuststokes}.
The approach is based on another Strang-type argument, similar to the pressure robust error analysis in \cite{linke2014role, 2016arXiv160903701L}. Like in section \ref{analysis:recon}, we also address the analysis of the the subsequent solution of \eqref{eq:probuststokes} given by an application of the reconstruction $\Recon$.
\begin{lemma} [Velocity consistency] \label{lem:velconsistency}
Let $\VelVarEx \in \Ureg$ be the velocity solution to \eqref{eq:stokes}.
For $\VelTest \in \VelSpace$ there holds
\begin{align*}
  \bilinearform{A}(u,v_h) & = \int_\Omega -\nu \Delta u  v_ h \, dx -  \ConDiff_{\VelVarEx}(\VelVarEx, \VelTest) \\
\text{with} \quad  \ConDiff_{\VelVarEx}(\VelVarEx, \VelTest) & := \sum_{T\in\mesh} \int_{\partial T} (\id - \facetproj^{k-1}) (- \nu \frac{\partial \VelVarEx}{\partial n}) \cdot (\id - \facetproj^{k-1}) v_T \, ds.
\end{align*}
\end{lemma}
\vspace*{-0.25cm}
\begin{proof}
This follows similar lines as the proof of Lemma \ref{lem:consistency}.
\end{proof}
\begin{theorem}{} \label{theom:errestprobust}
  Let $\VelVarEx \in  [H_0^1(\Omega) \cap H^l(\mesh)]^d,~l\geq 2$ be the velocity solution of~\eqref{eq:stokes} and $\VelVar \in \VelSpace$ be the velocity solution of \eqref{eq:probuststokes} where $\Recon$ fulfills Assumption \ref{ass:recon}.
Then there holds $\ReconHdiv \HdivVar \in H(\divergence,\Omega)$, $\divergence(\ReconHdiv \HdivVar) = 0$ and for $m = \min\{l-1, k\}$ we have
\begin{subequations}
  \begin{align}
    \brokenHnormleft \VelVarEx - \VelVar  \brokenHnormright  & \le  h^m \|  \VelVarEx \|_{H^{m+1}(\mesh)}, \label{eq:presta}\\
    \brokenHnormleft \VelVarEx - \Recon \VelVar  \brokenHnormright  & \le  h^m \|  \VelVarEx \|_{H^{m+1}(\mesh)}. \label{eq:prestb}
  \end{align}
\end{subequations}
\end{theorem}
\vspace*{-0.25cm}
\begin{proof}
  We start with \eqref{eq:presta}.
  For an arbitrary $v_h \in \VelSpace^0:=\{ v_h \in U_h: \bilinearform{B}(v_h,q_h)=0 ~ \forall q_h \in \PressureSpace \}$ we have
  $\brokenHnorm{u - u_h} \leq     \brokenHnorm{u_h - v_h} + 
  \brokenHnorm{u - v_h}$. Using Lemma \ref{lem:coercivity} and \ref{lem:continuity} and $w_h:=u_h-v_h$ leads to
  \begin{subequations}
    \begin{align*}
       \nu \brokenHnorm{w_h}^2
      & \lesssim \bilinearform{A}(u_h-v_h,u_h-v_h) 
       = \bilinearform{A}(u-v_h,w_h) + \bilinearform{A}(u_h-u,w_h) \\
      & \lesssim \nu \brokenHnormII{u - v_h} \brokenHnorm{w_h} + \bilinearform{A}(u_h,w_h) - \bilinearform{A}(u,w_h) \\
      & \lesssim \nu \brokenHnormII{u - v_h} \brokenHnorm{w_h} + f(\Recon w_h )-\bilinearform{B}(w_h,p_h) - \bilinearform{A}(u,w_h).
    \end{align*}
    As $\Recon w_h = \Recon{u_h} + \Recon{v_h}$ is exactly divergence free we have $f(\Recon w_h ) = (-\Delta u, \ReconHdiv w_{\mathcal{T}} )$ and further $w_{\mathcal{T}}$ is element-wise divergence free so $\bilinearform{B}(w_h,p_h) = 0$. Using Lemma \ref{lem:velconsistency} for the last term and
   property \eqref{eq:ass2b} we have for all $\delta_h \in [\mathcal{P}^{k-2}(T)]^d$ 
    \begin{align*}
      \nu \brokenHnorm{w_h}^2& \lesssim \brokenHnormII{u - v_h} \brokenHnorm{w_h} + (- \nu\Delta u , \ReconHdiv w_{\mathcal{T}} - w_{\mathcal{T}})_\Omega +  \ConDiff_{\VelVarEx}(\VelVarEx, w_h)  \\
      & \lesssim           \nu \brokenHnormII{u - v_h} \brokenHnorm{w_h} + \nu \Vert \Delta u- \delta_h \Vert_{L^2(\Omega)} \ \Vert \ReconHdiv w_{\mathcal{T}} - w_{\mathcal{T}} \Vert_{L^2(\Omega)} +  \ConDiff_{\VelVarEx}(\VelVarEx, w_h).\\
      \intertext{Next note that the operator $ (\id - \ReconHdiv)$ is orthogonal on constants so by using the Bramble-Hilbert Lemma, choosing $\delta_h = \Pi^{k-2} (\Delta u)$ and a scaling argument it follows}    
      \nu \brokenHnorm{w_h}^2& \lesssim           \nu \brokenHnormII{u - v_h} \brokenHnorm{w_h} +  \nu h^{m-1} \Vert u  \Vert_{H^{m+1}(\mesh)} \ h \brokenHnorm{w_h} + \nu h^{m} \Vert u  \Vert_{H^{m+1}(\mesh)} \brokenHnormleft w_h \brokenHnormright
     % ...
    \end{align*}
    Dividing by $\brokenHnormleft w_h \brokenHnormright$ finally leads to
    \begin{align*}
      \brokenHnorm{u - u_h} \lesssim  \inf\limits_{v_h \in U_h^0} \brokenHnormII{u - v_h} + h^{m} \Vert u  \Vert_{H^{m+1}(\mesh)}.
    \end{align*}
    The first term is then estimated by the infimum over the set of exact divergence free, thus normal continuous functions $\{v_h \in W_h \times F_h: \bilinearform{B}(v_h,q_h)=0 ~ \forall q_h \in \PressureSpace\}$, and by using remark 4.10 from \cite{braess} and the Bramble-Hilbert Lemma we conclude the proof of \eqref{eq:presta}.
  \end{subequations}
  Finally, \eqref{eq:prestb} follows from $\brokenHnorm{\cdot} \leq \brokenHnormII{\cdot}$,  Lemma \ref{lem:reconapprox} and \eqref{eq:presta}.
\end{proof}


\section{Numerical examples} \label{sec:numex}
In this section we investigate the numerical solutions of the Stokes problem on $\Omega = [0,1]^d$ and a right hand side $f := -\nu \Delta u - \nabla p$ with the exact solution
\begin{align*}
&u = \curl \zeta \quad \text{with} \quad \zeta:=x^2(x-1)^2y^2(y-1)^2 \quad \text{and} \quad p:= x^5+y^5 - \frac{1}{3} \quad \text{for} \quad d = 2, \\
&u = \curl (\zeta,\zeta,\zeta) \quad \text{with} \quad \zeta:=x^2(x-1)^2y^2(y-1)^2z^2(z-1)^2 \\
  &~~~~~~~~~~~~~~~~~~~~~~~~~~~~~~~~~~~~~~~~~~~~~~~~~~~~~~~\text{and} \quad p:= x^5+y^5+z^5 - \frac{1}{2} \quad \text{for} \quad d = 3,
\end{align*}
and different values of $\nu$. All implementations of the numerical examples where performed within the finite element library NGSolve/Netgen, see \cite{netgen,schoeberl2014cpp11}.
\subsection{Pressure robust error estimates}  \label{sec:numex1}
As addressed in section \ref{hdgnse:stokes:pressurerobust} we verify the impact of pressure robustness of the proposed methods. For this we compare different polynomial orders and distinguish between the solution $\VelVar^{\text{\tiny B}}$ of problem \eqref{eq:discstokes}, the pressure-robust solution $\VelVar^{\text{\tiny PR}}$ of problem \eqref{eq:probuststokes} and their exact divergence free reconstructions $\Recon  \VelVar^{\text{\tiny B}}$ and $\Recon  \VelVar^{\text{\tiny PR}}$. We further choose a fixed viscosity $\nu=1e^{-3}$.
Figure \ref{fig:twodexampleplot}
shows the $L^2$ error and the error of the gradient for the two dimensional case.  Both methods \eqref{eq:discstokes} and \eqref{eq:probuststokes} lead to optimal convergence orders as expected by the theory see Proposition \ref{prop:errest} and Theorem \ref{theom:errestprobust} but the order of magnitude is much larger for the solution of problem \eqref{eq:discstokes}. Note that the scaling between the respective values is always approximately $1/\nu$. Furthermore we have a proper convergence order of the errors with respect to the exactly divergence free reconstructed solutions $\ReconHdiv  \HdivVar^{\text{\tiny B}}$ and $\ReconHdiv  \HdivVar^{\text{\tiny PR}}$ as expected by Theorem \ref{theom:errestprobust}. Where the plots of the $L^2$ error of the reconstructed solutions and their discrete divergence free counterparts nearly can't be distinguished when the polynomial order is increased we can see a very small difference in broken $H^1$ (semi) norm. Similar conclusions can be done for the three dimensional case, see Figure \ref{fig:twodexampleplot}.
To clarify the consequences of pressure-robustness, Figure \ref{fig:probustconvplot} shows the broken $H^1$ (semi) norm error for different values of $\nu = 10^j$ with $j = -6 \dots 2$. For $\nu \ge 1$ the irrotational part of the right hand side is of the same magnitude as the divergence free part, thus the errors $|| \nabla \VelVarEx - \nabla \HdivVar^{\text{\tiny B}}  ||_0$ and $|| \nabla \VelVarEx - \nabla\HdivVar^{\text{\tiny PR}} ||_0$ are in a close range. When $\nu$ gets smaller the irrotaional part of $f$ tends to dominate resulting in a bad velocity approximation, thus a big error $|| \nabla \VelVarEx - \nabla \HdivVar^{\text{\tiny B}}  ||_0$. In contrast, the pressure-robust solution is not affected by the decreasing viscosity and the error $|| \nabla \VelVarEx - \nabla\HdivVar^{\text{\tiny PR}} ||_0$ stays constant.

\begin{figure}[h]
  \begin{center}
    \input{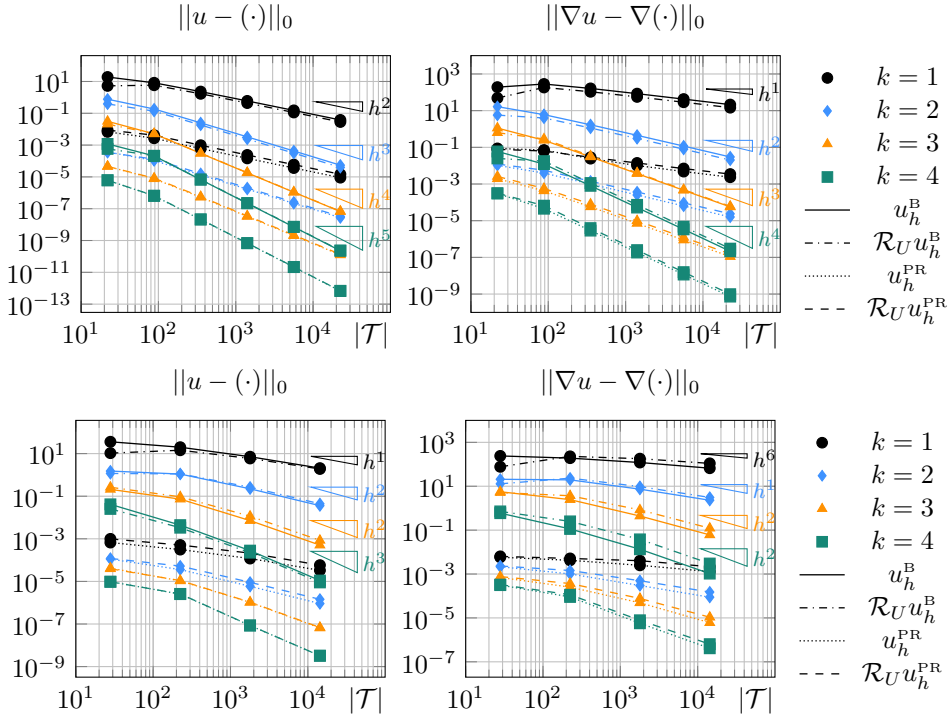}
    \pgfplotstableread{plots/probust/3dexample/p1_norec_nosolrec.out} \onenorecnosolrec
\pgfplotstableread{plots/probust/3dexample/p1_norec_solrec.out} \onenorecsolrec
\pgfplotstableread{plots/probust/3dexample/p1_rec_nosolrec.out} \onerecnosolrec
\pgfplotstableread{plots/probust/3dexample/p1_rec_solrec.out} \onerecsolrec
\pgfplotstableread{plots/probust/3dexample/p2_norec_nosolrec.out} \twonorecnosolrec
\pgfplotstableread{plots/probust/3dexample/p2_norec_solrec.out} \twonorecsolrec
\pgfplotstableread{plots/probust/3dexample/p2_rec_nosolrec.out} \tworecnosolrec
\pgfplotstableread{plots/probust/3dexample/p2_rec_solrec.out} \tworecsolrec
\pgfplotstableread{plots/probust/3dexample/p3_norec_nosolrec.out} \threenorecnosolrec
\pgfplotstableread{plots/probust/3dexample/p3_norec_solrec.out} \threenorecsolrec
\pgfplotstableread{plots/probust/3dexample/p3_rec_nosolrec.out} \threerecnosolrec
\pgfplotstableread{plots/probust/3dexample/p3_rec_solrec.out} \threerecsolrec
\pgfplotstableread{plots/probust/3dexample/p4_norec_nosolrec.out} \fournorecnosolrec
\pgfplotstableread{plots/probust/3dexample/p4_norec_solrec.out} \fournorecsolrec
\pgfplotstableread{plots/probust/3dexample/p4_rec_nosolrec.out} \fourrecnosolrec
\pgfplotstableread{plots/probust/3dexample/p4_rec_solrec.out} \fourrecsolrec

\definecolor{myblue}{RGB}{62,146,255}
\definecolor{mygreen}{RGB}{22,135,118}
\definecolor{myred}{RGB}{255,145,0}

\begin{tikzpicture}
  [
  scale=1
  ]
  \begin{axis}[
    name=plot1,
    scale=0.6,
    title = $|| u - (\cdot ) ||_0$,
    legend columns=4,
    legend style={text height=0.7em },
    legend style={draw=none},
    style={column sep=0.1cm},
    xlabel=$|\mathcal{T}|$,
    x label style={at={(0.85,0.065)},anchor=west},
    xmode=log,
    ymode=log,
    ytick = {1e-13,1e-11,1e-9,1e-7,1e-5,1e-3,1e-1,1e1},
    y tick label style={
      /pgf/number format/.cd,
      fixed,
      precision=2
    },
    xtick = {1e1,1e2,1e3,1e4},
    x tick label style={
      /pgf/number format/.cd,
      fixed,
      precision=2
    },
    %
    xmin = 1e1,
    xmax = 1e5,
    grid=both,
    legend style={
      cells={align=left},
      at={(0.3,-0.25)},
      anchor = north west
    },
    ]

    \addlegendimage{only marks, mark=*,black}
    \addlegendimage{only marks, mark=diamond*,myblue}
    \addlegendimage{only marks, mark=triangle*,myred}
    \addlegendimage{only marks, mark=square*,mygreen}
    \addlegendimage{line width=0.5pt}
    \addlegendimage{dashdotted,line width=0.5pt}
    \addlegendimage{densely dotted,line width=0.5pt}
    \addlegendimage{dashed,line width=0.5pt}

    
    \addplot[line width=0.5pt, color=black, mark=*] table[x=0, y=2]{\onenorecnosolrec};
    \addplot[line width=0.5pt, dashdotted, color=black, mark=*, mark options ={solid}] table[x=0, y=2]{\onenorecsolrec};
    \addplot[line width=0.5pt, densely dotted, color=black, mark=*, mark options ={solid}] table[x=0, y=2]{\onerecnosolrec};
    \addplot[line width=0.5pt, dashed, color=black, mark=*, mark options ={solid}] table[x=0, y=2]{\onerecsolrec};

    \addplot[line width=0.5pt, color=myblue, mark=diamond*] table[x=0, y=2]{\twonorecnosolrec};
    \addplot[line width=0.5pt, dashdotted, color=myblue, mark=diamond*, mark options ={solid}] table[x=0, y=2]{\twonorecsolrec};
    \addplot[line width=0.5pt, densely dotted, color=myblue, mark=diamond*, mark options ={solid}] table[x=0, y=2]{\tworecnosolrec};
    \addplot[line width=0.5pt, dashed, color=myblue, mark=diamond*, mark options ={solid}] table[x=0, y=2]{\tworecsolrec};

    \addplot[line width=0.5pt, color=myred, mark=triangle*] table[x=0, y=2]{\threenorecnosolrec};
    \addplot[line width=0.5pt, dashdotted, color=myred, mark=triangle*, mark options ={solid}] table[x=0, y=2]{\threenorecsolrec};
    \addplot[line width=0.5pt, densely dotted, color=myred, mark=triangle*, mark options ={solid}] table[x=0, y=2]{\threerecnosolrec};
    \addplot[line width=0.5pt, dashed, color=myred, mark=triangle*, mark options ={solid}] table[x=0, y=2]{\threerecsolrec};

    \addplot[line width=0.5pt, color=mygreen, mark=square*] table[x=0, y=2]{\fournorecnosolrec};
    \addplot[line width=0.5pt, dashdotted, color=mygreen, mark=square*, mark options ={solid}] table[x=0, y=2]{\fournorecsolrec};
    \addplot[line width=0.5pt, densely dotted, color=mygreen, mark=square*, mark options ={solid}] table[x=0, y=2]{\fourrecnosolrec};
    \addplot[line width=0.5pt, dashed, color=mygreen, mark=square*, mark options ={solid}] table[x=0, y=2]{\fourrecsolrec};

    \logLogSlopeTriangle{0.91}{0.15}{0.86}{0.66666666}{black};
    \logLogSlopeTriangle{0.91}{0.15}{0.74}{1}{myblue};
    \logLogSlopeTriangle{0.91}{0.15}{0.61}{1.33333333}{myred};
    \logLogSlopeTriangle{0.91}{0.15}{0.49}{1.66666666}{mygreen};

  \end{axis}
   \begin{axis}[
     name=plot2,
     at=(plot1.right of south east),
     anchor=left of south west,
     title = $|| \nabla u - \nabla (\cdot) ||_0$,
     legend entries={$k=1$, $k=2$, $k=3$, $k=4$, $\VelVar^{\text{\tiny B}}$, $\Recon \VelVar^{\text{\tiny B}}$, $ \VelVar^{\text{\tiny PR}}$, $\Recon  \VelVar^{\text{\tiny PR}}$  },                             
    scale=0.6,
    legend style={text height=0.7em },
    legend style={draw=none},
    style={column sep=0.15cm},
    xlabel=$|\mathcal{T}|$,
    x label style={at={(0.85,0.065)},anchor=west},
    xmode=log,
    ymax=1e4,
    ymin=2e-8,
    ymode=log,
    ytick = {1e-7,1e-5,1e-3,1e-1,1e1,1e3},
    y tick label style={
      /pgf/number format/.cd,
      fixed,
      precision=2
    },
    xtick = {1e1,1e2,1e3,1e4},
    x tick label style={
      /pgf/number format/.cd,
      fixed,
      precision=2
    },
    %
    xmin = 1e1,
    xmax = 1e5,
    grid=both,
    legend style={
      cells={align=left},
      at={(1.05,1)},
      anchor = north west
    },
    ]
    \addlegendimage{only marks, mark=*,black}
    \addlegendimage{only marks, mark=diamond*,myblue}
    \addlegendimage{only marks, mark=triangle*,myred}
    \addlegendimage{only marks, mark=square*,mygreen}
    \addlegendimage{line width=0.5pt}
    \addlegendimage{dashdotted,line width=0.5pt}
    \addlegendimage{densely dotted,line width=0.5pt}
    \addlegendimage{dashed,line width=0.5pt}
    
    \addplot[line width=0.5pt, color=black, mark=*] table[x=0, y=3]{\onenorecnosolrec};
    \addplot[line width=0.5pt, dashdotted, color=black, mark=*, mark options ={solid}] table[x=0, y=3]{\onenorecsolrec};
    \addplot[line width=0.5pt, densely dotted, color=black, mark=*, mark options ={solid}] table[x=0, y=3]{\onerecnosolrec};
    \addplot[line width=0.5pt, dashed, color=black, mark=*, mark options ={solid}] table[x=0, y=3]{\onerecsolrec};
                        
    \addplot[line width=0.5pt, color=myblue, mark=diamond*] table[x=0, y=3]{\twonorecnosolrec};
    \addplot[line width=0.5pt, dashdotted, color=myblue, mark=diamond*, mark options ={solid}] table[x=0, y=3]{\twonorecsolrec};
    \addplot[line width=0.5pt, densely dotted, color=myblue, mark=diamond*, mark options ={solid}] table[x=0, y=3]{\tworecnosolrec};
    \addplot[line width=0.5pt, dashed, color=myblue, mark=diamond*, mark options ={solid}] table[x=0, y=3]{\tworecsolrec};
                        
    \addplot[line width=0.5pt, color=myred, mark=triangle*] table[x=0, y=3]{\threenorecnosolrec};
    \addplot[line width=0.5pt, dashdotted, color=myred, mark=triangle*, mark options ={solid}] table[x=0, y=3]{\threenorecsolrec};
    \addplot[line width=0.5pt, densely dotted, color=myred, mark=triangle*, mark options ={solid}] table[x=0, y=3]{\threerecnosolrec};
    \addplot[line width=0.5pt, dashed, color=myred, mark=triangle*, mark options ={solid}] table[x=0, y=3]{\threerecsolrec};
                        
    \addplot[line width=0.5pt, color=mygreen, mark=square*] table[x=0, y=3]{\fournorecnosolrec};
    \addplot[line width=0.5pt, dashdotted, color=mygreen, mark=square*, mark options ={solid}] table[x=0, y=3]{\fournorecsolrec};
    \addplot[line width=0.5pt, densely dotted, color=mygreen, mark=square*, mark options ={solid}] table[x=0, y=3]{\fourrecnosolrec};
    \addplot[line width=0.5pt, dashed, color=mygreen, mark=square*, mark options ={solid}] table[x=0, y=3]{\fourrecsolrec};
    

    \logLogSlopeTriangle{0.91}{0.15}{0.87}{0.3333333}{black};
    \logLogSlopeTriangle{0.91}{0.15}{0.75}{0.6666666}{myblue};
    \logLogSlopeTriangle{0.91}{0.15}{0.63}{1}{myred};
    \logLogSlopeTriangle{0.91}{0.15}{0.5}{1.3333333}{mygreen};    
  \end{axis}
\end{tikzpicture}

  \end{center}
  \vspace*{-0.35cm}
  \caption{$L^2$ norm and the $H^1$ (semi) norm error for $d=2$ (top) and $d=3$ (bottom).}   \vspace*{-0.05cm}
\label{fig:twodexampleplot}
\end{figure}


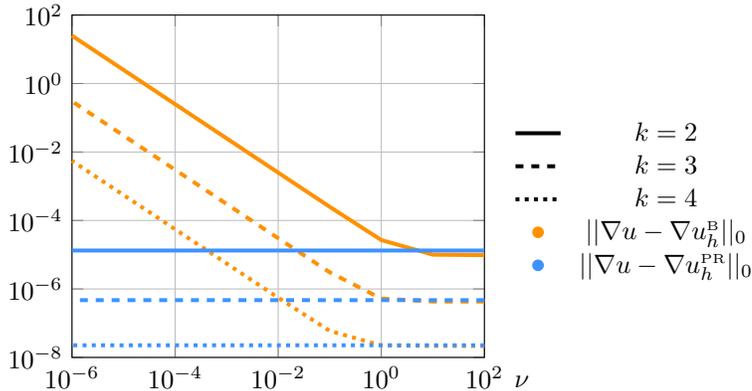
\begin{figure}[h]
  \begin{center}



\pgfplotstableread{plots/probust/varnu_2d_2hodc.out} \varnutwohodc
\pgfplotstableread{plots/probust/varnu_2d_2hodc_rec.out} \varnutwohodcrec

\pgfplotstableread{plots/probust/varnu_2d_3hodc.out} \varnuthreehodc
\pgfplotstableread{plots/probust/varnu_2d_3hodc_rec.out} \varnuthreehodcrec

\pgfplotstableread{plots/probust/varnu_2d_4hodc.out} \varnufourhodc
\pgfplotstableread{plots/probust/varnu_2d_4hodc_rec.out} \varnufourhodcrec

\definecolor{myblue}{RGB}{62,146,255}
\definecolor{mygreen}{RGB}{22,135,118}
\definecolor{myred}{RGB}{255,145,0}

\begin{tikzpicture}
  [
  scale=1
  ]
  \begin{axis}[
    name=plot1,
    legend entries={$k=2$, $k=3$, $k=4$, $|| \nabla \VelVarEx - \nabla \VelVar^{\text{\tiny B}}  ||_0$, $|| \nabla \VelVarEx - \nabla \VelVar^{\text{\tiny PR}}  ||_0$},                             
    scale=0.8,
    legend style={text height=0.7em },
    legend style={draw=none},
    style={column sep=0.15cm},
    xlabel=$\nu$,
    x label style={at={(1.05,0.045)},anchor=west},
    xmode=log,
    ymax=1e2,
    ymin=1e-8,
    ymode=log,
    ytick = {1e-8, 1e-6, 1e-4 , 1e-2, 1e0, 1e2},
    y tick label style={
      /pgf/number format/.cd,
      fixed,
      precision=2
    },
    xtick = {1e-6, 1e-4, 1e-2 , 1e0, 1e2},
    x tick label style={
      /pgf/number format/.cd,
      fixed,
      precision=2
    },
    %
    xmin = 1e-6,
    xmax = 1e2,
    xtick={1e-6,1e-4,1e-2,1e-0,1e2},
    grid=both,
    legend style={
      cells={align=left},
      at={(1.05,0.72)},
      anchor = north west
    },
    ]
    \addlegendimage{line width=1.5pt}
    \addlegendimage{line width=1.5pt,dashed}
    \addlegendimage{line width=1.5pt,dotted} 
    
    \addlegendimage{only marks, mark=*,myred}
    \addlegendimage{only marks, mark=*,myblue}
    
        

    \addplot[line width=1.5pt, color=myred] table[x=0, y=2]{\varnutwohodc};
    \addplot[line width=1.5pt,dashed, color=myred] table[x=0, y=2]{\varnuthreehodc};
    \addplot[line width=1.5pt, dotted, color=myred] table[x=0, y=2]{\varnufourhodc};
    

    \addplot[line width=1.5pt,  color=myblue] table[x=0, y=2]{\varnutwohodcrec};
    \addplot[line width=1.5pt,dashed, color=myblue] table[x=0, y=2]{\varnuthreehodcrec};
    \addplot[line width=1.5pt,dotted, color=myblue] table[x=0, y=2]{\varnufourhodcrec};
    
    
  \end{axis}

\end{tikzpicture}

  \end{center}
  \vspace*{-0.35cm}
  \caption{The broken $H^1$ (semi) norm errors for varying $\nu$ and different polynomial orders $k$.} \vspace*{-0.05cm}
\label{fig:probustconvplot}
\end{figure}

\subsection{Timings}
We discuss the computational benefits of the method on an example.
To this end we compare three methods:
First, the $H(\divergence)$-conforming HDG method introduced in \cite{lehrenfeld2010hybrid} with unknowns of order $k$ involved in the normal-continuity and the (weak) tangential continuity through facet unknowns (method \kk). Secondly, the $H(\divergence)$-conforming HDG method with reduced stabilization (projected jump) where (only) the tangential facet unknowns are reduced to order $k-1$, as presented in \cite{LS_CMAME_2016} (method \kkmo). Thirdly, the method that we introduced in this contribution with a relaxed $H(\divergence)$-conformity so that the only globally coupled velocity unknowns are facet unknowns of order $k-1$ (method \kmokmo).
Note that the costs for the reconstruction steps for solenoidal solutions and pressure robustness are negligible and are not discussed here. For a comparison between the computational effort of an $H(\divergence)$-conforming DG and a corresponding HDG method we refer to \cite{LS_CMAME_2016}.

The following measures have been taken to compare the methods' computational effort.
We consider the number of (velocity and pressure) unknowns that appear in the three methods (\texttt{dofs}). Note that, in order to simplify the presentation, we do not make use of a possible a priori reduction of the velocity and pressure spaces, cf. Remark \ref{rem:hodivfree}.
Prior to solving the linear systems we apply static condensation, i.e. we eliminate all unknowns that have only element-local couplings. The remaing unknowns are denoted as the \emph{globally coupled} degrees of freedom (\texttt{gdofs}) which already give a better indication of the computational effort as these are the only ones for which a large linear system has to be solved.
We note that all pressure unknowns except for one constant per element can be eliminated, i.e. the pressure unknowns in the condensated system are independent of $k$. 
Not only the number of (globally coupled) unknowns is important for the computational effort in linear solvers, but also the sparsity pattern and the number of non-zero entries in the matrix (\texttt{nze}).
Finally, we use a direct factorization method to prepare the solution of linear systems and measure the computation time on a shared-memory machine with $24$ cores for the three methods (\texttt{f.time}). We note that the factorization time is much larger than forward and backward substitution which we neglect in the timings.
As examples we take the finest meshes from the 2D and 3D Stokes problems from the previous section. The results are shown in Table \ref{tab:timings}.

\begin{table}[h]
  \begin{center}
    \footnotesize
    \begin{tabular}{r r@{~}r@{~}r  r@{~}r@{~}r  r@{~}r@{~}r  r@{~}r@{~}r }
      \toprule
      & \multicolumn{3}{c}{$k=1$}
      & \multicolumn{3}{c}{$k=2$}
      & \multicolumn{3}{c}{$k=3$}
      & \multicolumn{3}{c}{$k=4$} \\
 2D   & \kmokmo & \kkmo & \kk 
      & \kmokmo & \kkmo & \kk 
      & \kmokmo & \kkmo & \kk 
      & \kmokmo & \kkmo & \kk \\
      \midrule
      \texttt{dofs}
      & \numQ{248794} &\numQ{215258} & \numQ{181722}
      & \numQ{429530} &\numQ{395994} & \numQ{362458}
      & \numQ{677850} &\numQ{644314} & \numQ{610778}
      & \numQ{993754} &\numQ{960218} & \numQ{926682} \\
      \texttt{gdofs}
      & \numQ{90624} &\numQ{124672} & \numQ{158720}
      & \numQ{158720} &\numQ{192768} & \numQ{226816}
      & \numQ{226816} &\numQ{260864} & \numQ{294912}
      & \numQ{294912} &\numQ{328960} & \numQ{363008}\\
      \texttt{nze}
      & \numQ{969728} & \numQQ{1950976} & \numQQ{3270656} 
      & \numQQ{3270656} & \numQQ{4928768} & \numQQ{6925312} 
      & \numQQ{6925312} & \numQQ{9260288} & \numQQ{11933696} 
      & \numQQ{11933696} & \numQQ{14945536} & \numQQ{18295808} \\
      \texttt{f.time}
      & \numt{0.25324320793151855} &\numt{0.7063219547271729} & \numt{1.2185072898864746}
      & \numt{0.8981707096099854 } &\numt{1.2550511360168457} & \numt{1.8931264877319336}
      & \numt{1.7866754531860352 } &\numt{2.343336582183838 } & \numt{3.0391082763671875}
      & \numt{3.149350881576538  } &\numt{3.930508852005005 } & \numt{4.976104259490967 }\\
      \midrule
      & \multicolumn{3}{c}{$k=1$}
      & \multicolumn{3}{c}{$k=2$}
      & \multicolumn{3}{c}{$k=3$}
      & \multicolumn{3}{c}{$k=4$} \\
 3D   & \kmokmo & \kkmo & \kk 
      & \kmokmo & \kkmo & \kk 
      & \kmokmo & \kkmo & \kk 
      & \kmokmo & \kkmo & \kk \\
      \midrule
      \texttt{dofs}
      & \numQ{461270} &\numQ{406230 } & \numQ{296150}
      & \numQ{941270} &\numQ{858710 } & \numQ{693590}
      & \numQQ{1668438} & \numQQ{1558358} & \numQQ{1338198}
      & \numQQ{2700118} & \numQQ{2562518} & \numQQ{2287318} \\
      \texttt{gdofs}
      &  \numQ{103808} &\numQ{163456} & \numQ{282752}
      & \numQ{282752} &\numQ{372224} & \numQ{551168}
      & \numQ{551168} &\numQ{670464} & \numQ{909056}
      & \numQ{909056} &\numQQ{1058176} & \numQQ{1356416} \\
      \texttt{nze}
      &  \numQQ{2175104} & \numQQ{5634176} & \numQ{17396864}
      &  \numQ{17396864} & \numQ{30457856} & \numQ{67480064}
      &  \numQ{67480064} & \numQ{100235776} &\numQ{185125376}
      & \numQ{185125376} &\numQ{251302016} &\numQ{413933696} \\
      \texttt{f.time}
      & \numt{0.6871836185455322} &\numt{1.2109129428863525} &\numt{4.4992995262146}
      & \numt{4.4400787353515625} &\numt{7.6241114139556885} &\numt{18.996856212615967}
      & \numt{18.757426738739014} &\numt{31.923074960708618} &\numt{73.89999485015869}
      & \numt{73.94802737236023} &\numt{114} &\numt{227} \\
      \bottomrule
    \end{tabular}
    \caption{Measures for the costs of solving linear systems for different $H(\divergence)$-conforming HDG discretization methods for the (2D and 3D) test cases from Section \ref{sec:numex1}.} 
    \label{tab:timings}
  \end{center}
  \vspace*{-0.5cm}
\end{table}

First, we observe that the number of unknowns for the methods with relaxed (weak) continuity are larger due to break up of a globally coupled basis functions into two only locally coupled unknowns. Accordingly the globally coupled unknowns are much less for these methods which also decreases the number of non-zero entries in the Schur complement and thus the computing time for the factorization.

The computational effort spend in the linear systems for \kk~ for some order $k$ is as muss as \kmokmo~ with $k+1$, i.e. we obtain an increase in accuracy (cf. previous section) with the cost of a method of order $k$ only. The benefit is especially strong for low orders $k$.
Let us mention that the time to form the Schur complement and later to reconstruct the interior unknowns is negligible compared to the costs of solving the Schur complement. Further, both operations are element-local and can easily be done in parallel.

There is no consistent DG discretization with $k=0$. However, after static condensation, for \kmokmo~ and $k = 1$, only unknowns of degree $0$ remain for the velocity, i.e. one degree of freedom per facet and dimension. 



\FloatBarrier

\ifarxiv
\appendix
\section*{Appendix}
\subsection{An equivalent formulation to \eqref{eq:discstokes} based on scalar FE spaces} \label{hdgnse:stokes:fesalternative}
We present a different discrete formulation which results in the same discrete solution but uses the much simpler finite element spaces for velocity and pressure:
\begin{equation*}
  V_h =  \prod_{T \in \mesh} [\mathcal{P}^{k}(T)]^d, \quad\Lambda_h = \prod_{F \in \mathcal{F}} [\mathcal{P}^{k-1}(F)]^d  \quad\text{ and }\quad Q_h =  \prod_{T \in \mesh} \mathcal{P}^{k-1}(T).
\end{equation*}
We note that the velocity spaces are product spaces of scalar-valued finite element spaces for element and facet unknowns without any continuity between elements.
The facet space $\Lambda_h$ is responsible for two tasks:
The tangential components of the function in $\Lambda_h$ take the role of the facet variables in $\FacetSpace$ in \eqref{eq:discstokes}. The normal component is a Lagrangian multiplier introduced to enforce the weak $H(\divergence)$-conformity of solutions.

We define the bilinear forms 
\begin{align} \label{eq:blfAt}
    \widetilde{\bilinearform{A}}(\VelVar,\VelTest) :=
    & \displaystyle \sumoverallelements \int_{T} \nu {\nabla} {\HdivVar} \! : \! {\nabla} {\HdivTest} \ d {x} - \int_{\partial T} \nu \frac{\partial {\HdivVar}}{\partial {\normal} }  \facetproj^{k-1} \facetjump{ \VelTest^t } \ d {s} \\
    & \displaystyle- \int_{\partial T} \nu \frac{\partial {\HdivTest}}{\partial {\normal} } \facetproj^{k-1} \facetjump{ \VelVar^t } \ d {s}
      + \int_{\partial T} \nu \frac{\stab k^2}{h} \facetproj^{k-1} \facetjump{ \VelVar^t }  \facetproj^{k-1} \facetjump{ \VelTest^t } \ d {s}, \nonumber
\intertext{for $u_h = (u_T,u_F), v_h = (v_T,v_F) \in V_h \times \Lambda_h$ and for $\VelVar,\VelTest \in V_h \times \Lambda_h, \ \PressureVar \in \Lambda_h$}
\label{eq:blfBt}
      \widetilde{\bilinearform{B}}({\VelVar},(\VelTest,\PressureVar))
      := & \sumoverallelements - \int_{T} \PressureVar
  \divergence({\HdivVar}) \ d {x} + \int_{\partial T} u_T v_F^n \, ds.
\end{align}
Then the alternative discretization takes the form:
Find $ \VelVar\in V_h \times \Lambda_h$, $\PressureVar \in Q_h$, s.t. for all
$\VelTest \in V_h \times \Lambda_h, \PressureTest \in Q_h$ there holds
\begin{equation}\label{eq:discstokes2}
  \widetilde{\bilinearform{A}}(\VelVar,{\VelTest})
  + \widetilde{\bilinearform{B}}({\VelTest},(\VelVar,\PressureVar))
  + \widetilde{\bilinearform{B}}(\VelVar,(\VelTest,\PressureTest))
  =
  \ForceVar( \VelTest ).
  \tag{$\tilde{\text{B}}$}
\end{equation}
The conditions
$\widetilde{\bilinearform{B}}(\VelVar,(0,\PressureTest)) = 0,\ \forall~ \PressureTest \in Q_h$
and $\widetilde{\bilinearform{B}}(\VelVar,((0,v_F^n),0)) = 0,\ \forall~ v_F \in \Lambda_h$
enforce divergence free solutions within every element $T$, $\divergence({u_T})=0$ and a relaxed normal continuity, $\facetproj^{k-1}\jumpDG{u \cdot n}=0$, respectively.
This implies that velocity solutions to \eqref{eq:discstokes2} will be in $\HdivSpaceHODC$. We note that
on the subspace of functions in $\HdivSpaceHODC \subset V_h$ we have that $\widetilde{\bilinearform{A}}(\cdot,\cdot) = \bilinearform{A}(\cdot,\cdot)$
and
$\widetilde{\bilinearform{B}}(\cdot,\cdot) = \bilinearform{B}(\cdot,\cdot)$ so that we obtain the same functions $u_T,~u_F^t$ and $p_h$ as solutions to \eqref{eq:discstokes} and \eqref{eq:discstokes2}.
The additional unknown $u_F^n$ is an approximation to the normal stress $p - \nu n^T\!\cdot\!\nabla u\!\cdot\! n$ as $u_F^n$ is the Lagrangian multiplier for the (weak) normal continuity.

We note that although the solution to the equivalent formulations \eqref{eq:discstokes} and \eqref{eq:discstokes2} are the same, the sparsity pattern is different, cf. \cite[Remark 7]{LS_CMAME_2016}.
Also, we note that while the implementation of the finite element spaces is easier in \eqref{eq:discstokes2}, the implementation of the projection $\facetproj^{k-1}$ is still required, cf. the discussion in the previous subsection.
Furthermore, the usual basis for the simpler space $V_h$ does not provide the simple and fast realization of the reconstruction operator that we discussed in section \ref{hdgnse:stokes:reconstruction:average} so that the reconstruction and the pressure robust variant \eqref{eq:probuststokes} will be more involved and computationally more expensive.

In \cite{rhebergenwells2016} a discretization similar to \eqref{eq:discstokes2} is considered for a Navier-Stokes problem, with the important difference that instead of one facet space $\Lambda_h$, two separate facet spaces $\Lambda_h^u$ and $\Lambda_h^p$ are used where $\Lambda_h^u = \prod_{F \in \mathcal{F}} [\mathcal{P}^{k}(F)]^d$ is the facet space for the momentum conservation in the velocity discretization (and thus no projection is needed) while $\Lambda_h^p = \prod_{F \in \mathcal{F}} \mathcal{P}^{k-1}(F)$ implements the weak $H(\divergence)$-conformity.
In a similar fashion different discretizations can be derived when varying different polynomial degrees in $\Lambda_h^u$, $\Lambda_h^p$ or $\Lambda_h$.
When choosing degree $k$ in $\Lambda_h^p$ and $\Lambda_h^u$ (and removing the projection $\facetproj^{k-1}$) we obtain yet another formulation  with divergence free solutions that has recently been proposed in \cite{rhebergenwells2017}.
\else
\fi

\bibliographystyle{siam}
\bibliography{literature}

\end{document}